\documentclass[10pt,regno]{amsart}
\usepackage{amsmath,amsthm,amsfonts,amssymb,amscd,overpic,mathrsfs}
\usepackage{euscript,graphicx,color}
\usepackage{epstopdf,rotating}
\newtheorem{mtheorem}{Theorem}
\newtheorem{mcorollary}[mtheorem]{Corollary}

\newtheorem{mfact}[mtheorem]{Fact}
\newtheorem{theorem}{Theorem}[section]
\newtheorem{lemma}[theorem]{Lemma}

\newtheorem{scholium}[theorem]{Scholium}
\newtheorem{corollary}[theorem]{Corollary}
\newtheorem{proposition}[theorem]{Proposition}

\theoremstyle{definition}
\newtheorem{definition}[theorem]{Definition}

\newtheorem{remark}[theorem]{Remark}
\newtheorem*{remark*}{Remark}

\numberwithin{equation}{section}

\newcommand{\eqdef}{\stackrel{\scriptscriptstyle\rm def}{=}}

\DeclareMathOperator{\clocon}{\overline{\rm conv}}

\DeclareMathOperator{\card}{card}
\DeclareMathOperator{\dist}{Dist}

\DeclareMathOperator{\interior}{int}

\DeclareMathOperator{\Mod}{Mod}
\def\D{\lVert F\rVert}
\def\Kdois{K_2}

\def\varkappa{\kappa}
\def\cA{\mathscr{A}}
\def\bN{\mathbb{N}}

\def\bZ{\mathbb{Z}}
\def\bR{\mathbb{R}}

\def\bS{\mathbb{S}}

\def\cS{\mathcal{S}}
\def\cT{\mathcal{T}}
\def\cU{\mathcal{U}}

\def\cV{\EuScript{V}}

\def\cM{\EuScript{M}}

\def\cL{\EuScript{L}}

\def\fX{\mathfrak{X}}

\def\fT{\mathfrak{T}}

\DeclareMathSymbol{\varnothing}{\mathord}{AMSb}{"3F}
\renewcommand{\emptyset}{\varnothing}

\begin{document}

\title[Nonhyperbolic step skew-products]{Nonhyperbolic step skew-products: \\ Ergodic approximation}
\author[L.~J.~D\'iaz]{L. J. D\'\i az}
\address{Departamento de Matem\'atica PUC-Rio, Marqu\^es de S\~ao Vicente 225, G\'avea, Rio de Janeiro 22451-900, Brazil}
\email{lodiaz@mat.puc-rio.br}
\author[K.~Gelfert]{K.~Gelfert}
\address{Instituto de Matem\'atica Universidade Federal do Rio de Janeiro, Av. Athos da Silveira Ramos 149, Cidade Universit\'aria - Ilha do Fund\~ao, Rio de Janeiro 21945-909,  Brazil}\email{gelfert@im.ufrj.br}
\author[M.~Rams]{M. Rams} \address{Institute of Mathematics, Polish Academy of Sciences, ul. \'{S}niadeckich 8,  00-656 Warszawa, Poland}
\email{rams@impan.pl}

\begin{abstract}
We study transitive step skew-product maps modeled over a complete shift of $k$, $k\ge2$, symbols whose fiber maps are defined on the circle and
have intermingled contracting  and expanding regions. These dynamics are genuinely nonhyperbolic and exhibit simultaneously ergodic measures with positive, negative, and zero exponents.

We introduce a set of axioms for the fiber maps and study the dynamics of the resulting skew-product. These axioms turn out to capture the key mechanisms of the dynamics
of nonhyperbolic robustly transitive maps with compact central leaves.

Focusing on the nonhyperbolic ergodic measures (with zero fiber exponent) of these systems, we prove that such measures are approximated in the weak$\ast$ topology and in entropy by hyperbolic ones.
We also prove that they are  
in the intersection of the  convex hulls of the measures with positive fiber exponent and
with negative fiber exponent.
Our methods also allow us to perturb hyperbolic measures. We can perturb a
 measure with
 negative exponent directly to a measure with positive exponent (and vice-versa), however
 we lose some amount of entropy in this process. The loss of entropy is determined by the difference between the Lyapunov exponents of the measures.
 \end{abstract}

\begin{thanks}{This research has been supported  [in part] by Pronex,
CNE-Faperj, and CNPq-grants (Brazil) and
EU Marie-Curie IRSES ``Brazilian-European partnership in Dynamical
Systems" (FP7-PEOPLE-2012-IRSES 318999 BREUDS). The authors 
acknowledge the hospitality of their home institutions. LJD was partially supported by  Palis-Balzan project.  MR was supported by National Science Centre grant 2014/13/B/ST1/01033 (Poland)}\end{thanks}

\maketitle
\tableofcontents

\begin{center}
\hfill \emph{ergodic measures}\\
\hfill \emph{perturb easily}\\
\hfill \emph{exponents rise or decrease}
\end{center}

\smallskip

\section{Introduction}

The aim of this paper is to understand the general structure and finer properties of the space of invariant measures of robustly transitive and robustly nonhyperbolic dynamical systems.
For a large class%
\footnote{open and dense for $C^1$ step skew-products and dense for general $C^1$ skew-products}
 of such skew-products we approximate in entropy and in the weak$\ast$ topology  ergodic measures which are nonhyperbolic (with a zero Lyapunov exponent) and have  positive entropy by measures supported on hyperbolic horseshoes, see Theorem~\ref{the:luzzaattoo}. 
 This result can be viewed as a
 nonhyperbolic version of a classical result by Katok\footnote{The result of Katok claims that
 any ergodic hyperbolic measure can be weak* and in entropy approximated by horseshoes.
  See~\cite{Kat:80,KatHas:95})  for $C^{1+\alpha}$ diffeomorphisms
and also  extensions in the context of $C^1$ diffeomorphisms with a dominated splitting in~\cite{Cro:11,LuzSan:13,Gel:}.} and also  as a partial answer to a question about abundance of hyperbolicity posed by Buzzi in~\cite[Section 1.5]{Buz:}\footnote{A bit more precisely, his question is the following:
Among partially hyperbolic diffeomorphisms with one-dimensional center direction,
are those with ``enough" hyperbolic measures $C^1$ or $C^2$ dense?}. As a consequence of our main results, in our setting,  Theorem~\ref{the:luzzaattoo} can be read as follows: the  intersection of the closed convex hull of ergodic measures with negative fiber exponent and the closed convex hull of ergodic measures with positive fiber exponent  is non-empty and contains all ergodic measures with zero exponent, see Corollary~\ref{cor:simplleexx}.

Our results are a step of a program to understand  the measure spaces, ergodic theory, and multifractal properties of general systems (diffeomorphisms, skew-product maps).
As this at the present state of the art is far too ambitious in  this vast generality, one may aim for gradually less specific classes of systems.%
\footnote{An example of this strategy can be found by the line of papers studying, in the same context, the construction of nonhyperbolic ergodic measures: \cite{GorIlyKleNal:05} (step skew-products),~\cite{KleNal:04} (skew-products and some specific open sets of diffeomorphisms), \cite{DiaGor:09,BonDiaGor:10} (generic diffeomorphisms), and~\cite{BocBonDia:} (settling open and densely the case of general diffeomorphisms).}
We  focus on partially hyperbolic systems.  The simplest, but still extremely complex, case occurs when the partially hyperbolic system has a central direction which is one-dimensional. Simplifying even one more step, in the case of  partially hyperbolic diffeomorphisms, one may assume that the central bundle is  integrable. In this case, three  different scenario can occur: there exist only non-compact leaves (DA -- derived from Anosov -- diffeomorphisms),  there exist simultaneously compact and non-compact leaves (time-1 maps of Anosov flows), or there exist only compact central leaves. The latter, and in some sense easiest, of these cases - compact central leaves - is still extremely rich  (see, for instance, the pathological behaviors of the central foliations in~\cite{RueWil:01,ShuWil:00}). On the other hand, using ingredients of one-dimensional dynamics, in this case one often has a very precise picture of the dynamics (see, for instance,~\cite{RodRodTahUre:12,KleVol:}).
As further simplification, we will  restrict ourselves to step skew-products over a complete shift with circle-fiber maps. We hope that one will be able to gradually carry this program to more general settings. In fact, it turns out that the systems studied in this paper cover already typical robustly transitive and nonhyperbolic skew-products (see Section~\ref{sss.robusttransitivity}).

Besides the fact that skew-products as a class of systems have an intrinsic interest (there is a vast literature about different aspects, we do not go into further details here), they can also serve as a first step on the way to understand general types of dynamics of diffeomorphisms or endomorphisms.
They also allow us to study essential aspects of a problem while escaping technical difficulties and this way enable us to study the problem in various steps of increasing difficulty.

To be a bit more precise, still in the partially hyperbolic setting with a nonhyperbolic
central direction, when aiming for general systems, one is confronted with several problems of completely different nature and origin.
First, restricting to systems with a one-dimensional central fiber enables us to study relatively easily their Lyapunov exponents which turn into Birkhoff averages of \emph{continuous} functions, while in the general case they are provided by the Oseledets theorem and are measurable functions only.
Moreover, in this case there is no entropy generated by the fiber dynamics (for details see Appendix).
A second problem is the nonhyperbolicity reflected by the coexistence of hyperbolic
measures and, consequently, of hyperbolic periodic points with different behaviour in the central direction. Finally, there are problems related to the existence and regularity of the central foliations.
Restricting our consideration to skew-products allows us to focus on the difficulty arising from the nonhyperbolicity, while escaping  from the latter one.
This approach also allows us to present our constructions (e.g. the multi-variable-time horseshoes and their symbolic extensions) in a transparent way. This strategy also allows us to establish an axiomatic approach, which is in fact completely justified and turns out to reflect quite well the general features of robustly transitive and nonhyperbolic systems.


Our axiomatic approach allows us to study the ergodic theory of step skew-products which mix expanding and contracting fiber dynamics. For instance, in the robustly transitive case, there are ``horseshoes'' which are contracting and expanding in the fiber direction, respectively, and there are also ergodic nonhyperbolic measures~\cite{GorIlyKleNal:05}, even with full support~\cite{BonDiaGor:10} or with positive entropy~\cite{BocBonDia:}. This shows that nonhyperbolic ergodic measures cannot be neglected.
Besides this, there is not much rigorous study of (the set of) points and measures with zero Lyapunov exponent by means of an analysis of the measure space and entropy properties.
For instance, as a consequence of~\cite{AbdBonCro:11} in our setting hyperbolic periodic  measures are dense%
\footnote{Indeed, this is true for an isolated homoclinic class of a $C^1$ diffeomorphism with a dominated splitting.}
and, in particular, every nonhyperbolic  ergodic measure is accumulated by hyperbolic ergodic ones. One of our key improvements is the approximation by entropy.


Our setting and one of our main applications is motivated by the study of
partially hyperbolic robustly transitive diffeomorphisms and minimality of their strong stable and unstable
foliations in~\cite{BonDiaUre:02,RodRodUre:07}. We extract some general principles which we put as a set of axioms. To be more precise, let $\sigma\colon\Sigma_k\to\Sigma_k$, $k\ge2$, be the usual shift map on the space $\Sigma_k=\{0,\ldots,k-1\}^\bZ$ of two-sided sequences. Consider a finite family $f_i\colon \bS^1\to \bS^1$, $i=0,\ldots,k-1$, of $C^1$
diffeomorphisms. Associated to these maps, we consider the step skew-product
\begin{equation}\label{eq:sp}
	F\colon \Sigma_k\times \bS^1\to \Sigma_k\times \bS^1,
	\quad
	F(\xi,x) = \big(\sigma(\xi), f_{\xi_0}(x)\big).
\end{equation}
Seeing the map as an iterated function system  (IFS) associated to the fiber maps $\{f_i\}_{i=0}^{k-1}$, we require that there is some ``expanding region'' (relative to the fiber direction)  and some ``contracting region'' and that any of these regions  ``can be reached'' from any point in the ambient space  under forward and backward iterations. 
More precisely, we say that the map $F$ satisfies Axioms CEC$\pm$ and Acc$\pm$ if  there is some nontrivial closed interval $J\subset\bS^1$ such that:
\begin{description}
\item[CEC$+$] \textbf{(Controlled Expanding forward Covering)}   Existence of some forward iteration of the fiber along which any small enough interval $H$ intersecting $J$ is uniformly expanded and covers $J$ (with uniform control on iteration length and expansion strength which depend on the size of $H$ only).
\item[CEC$-$] \textbf{(Controlled Expanding backward Covering)}
Axiom CEC$+$  for the IFS $\{f_i^{-1}\}$.
\item[Acc$+$] \textbf{(Forward Accessibility)} Forward iterations of $J$ cover $\bS^1$.
\item[Acc$-$]\textbf{(Backward Accessibility)} Axiom Acc$+$ for the IFS $\{f_i^{-1}\}$.
\end{description}
We call such an interval $J\subset\bS^1$  a \emph{blending interval}.

If the map $F$ is transitive and satisfies the axioms then  every sufficiently small interval is a blending interval,
see Section~\ref{sec:axioms} for  details and discussion.
For completeness, recall that $F$ is \emph{transitive} if for any pair of nonempty
 open sets $U,V\subset\Sigma_k\times\bS^1$ there is $n\ge1$ such that $F^n(U)\cap V\ne\emptyset$.

These axioms in particular imply that $F$ is \emph{robustly transitive}, that is, for every family of diffeomorphisms $g_0,\ldots,g_{k-1}$ $C^1$-close enough to $f_0,\ldots,f_{k-1}$ the resulting skew-product map $G$ is transitive and robustly nonhyperbolic (the spectrum of fiber Lyapunov exponents defined below is an interval containing $0$ in its interior).
We also observe that they appear naturally in robustly transitive step skew-products.
 The expanding/contracting regions reflect the co-existence of hyperbolic periodic points with different central behavior and can be identified with a so-called  ``expanding/contracting blender'', while the other properties reflect the minimality of the strong stable and unstable foliations (see Section~\ref{sec:examples}).

We are now ready to state our main result. Let $\cM$ be the space of $F$-invariant probability measures supported in $\Sigma_k\times \bS^1$.
Denote by $\cM_{\rm erg}\subset\cM$ the subset of ergodic measures. Given $\mu\in\cM_{\rm erg}$ denote by $\chi(\mu)$ its \emph{(fiber) Lyapunov exponent} which is given by
\begin{equation*}
	\chi(\mu)
	\eqdef \int\log\,\lvert(f_{\xi_0})'(x)\rvert\,d\mu(\xi,x). 	
\end{equation*}
A measure $\mu$ is called {\emph{nonhyperbolic}} if $\chi(\mu)=0$.
Otherwise the measure is called {\emph{hyperbolic}}.

Given a compact $F$-invariant set $\Gamma\subset\Sigma_k\times\bS^1$, we will denote by $\cM(\Gamma)\subset\cM$ the subset of measures supported in $\Gamma$. We equip this space with the weak$\ast$ topology. We say that $\Gamma$ has {\emph{uniform central expansion}}
(\emph{contraction}) if every ergodic measure $\mu\in \cM(\Gamma)$ has positive (negative) fiber Lyapunov exponent.

 We denote by $h_{\rm top}(F,\Gamma)$ the topological entropy of $F$ on $\Gamma$ and by $h(\mu)$ the \emph{entropy} of a measure $\mu$.

\begin{mtheorem}\label{the:luzzaattoo}
Consider a transitive step skew-product map $F$ as in~\eqref{eq:sp} whose fiber maps are $C^1$.
	Assume that $F$ satisfies Axioms CEC$\pm$ and Acc$\pm$.
	
	Then for every nonhyperbolic measure
	$\mu\in\cM_{\rm erg}$ ($\chi(\mu)=0$) 
	for every $\delta>0$ and every $\gamma>0$ there exist compact $F$-invariant  transitive hyperbolic sets $\Gamma^+$ with uniform central expansion
	and $\Gamma^-$ with uniform central contraction whose topological entropies satisfy
\[	
	h_{\rm top}(F,\Gamma^+),  h_{\rm top}(F,\Gamma^-) \in  [h(\mu)-\gamma, h(\mu)+\gamma].
\]	
Moreover, every measure $\nu^\pm \in\cM(\Gamma^\pm)$  is $\delta$-close to $\mu^\pm$ in the weak$\ast$ topology.

In particular, there are  hyperbolic measures $\nu^+, \nu^-\in\cM_{\rm erg}$ with 
$$
\chi(\nu^+)\in(0,\delta) \quad \text{and} \quad  \chi(\nu^-)\in(-\delta,0)
$$
and
$$
h(\nu^+), h(\nu^-)\in [h(\mu)-\gamma, h(\mu)+\gamma].
$$

If $h(\mu)=0$ then $\Gamma^-$ and $\Gamma^+$ are hyperbolic periodic orbits.
\end{mtheorem}

%
%

We will prove Theorem~\ref{the:luzzaattoo} only in the case $h(\mu)>0$. If $h(\mu)=0$ the same proof allows us to construct a periodic orbit (in the place of a
compact, $F$-invariant, hyperbolic, and transitive set $\Gamma^\pm$ with positive entropy,
according to the case) with the required Lyapunov exponent.

Investigating the structure of the space of invariant measures, this theorem can be stated in slightly different terms.
For that recall that the set $\cM$ equipped with the weak$\ast$ topology is a Choquet simplex, the ergodic measures are its extreme points, and any $\mu\in\cM$ has a unique ergodic decomposition (see~\cite[Chapter 6.2]{Wal:82}). In some contexts it can be shown that the set of ergodic measures $\cM_{\rm erg}$ is dense in its closed convex hull $\cM$ (in this case, if $\cM$ is non-trivial, this is called a \emph{Poulsen simplex}). In the general case, $\cM$ does not have such a property. However, in our setting, by~\cite{BonGel:} the subset of ergodic measures with positive fiber Lyapunov exponent (with negative fiber Lyapunov exponent) is indeed a Poulsen simplex. We  investigate further these simplices and study the remaining set of (ergodic) measures with zero fiber  exponent. We consider the decomposition
\begin{equation*}
	\cM_{\rm erg}=\cM_{\rm erg,<0}\cup\cM_{\rm erg,0}\cup\cM_{\rm erg,>0}
\end{equation*}
into ergodic measures with negative, zero, and positive (fiber)  Lyapunov exponent, respectively. We will sometimes also consider the corresponding spaces $\cM_{\rm erg,\le0}$ and $\cM_{\rm erg,\ge0}$.
We consider this decomposition as an important step towards the
study of ergodic theory of general nonhyperbolic systems.

\begin{mcorollary}\label{cor:simplleexx}
	Under the assumptions of Theorem~\ref{the:luzzaattoo}, the  intersection of the closed convex hull of ergodic measures with negative fiber exponent and the closed convex hull of ergodic measures with positive fiber exponent  is non-empty and contains all ergodic measures with zero exponent.
\end{mcorollary}

We observe that the axioms guarantee the existence of ``horseshoes" and therefore the map $F$ has positive topological entropy.
We have the following particular variational principle of entropy.

\begin{mtheorem}\label{theocor:varprinc}
Under the assumptions of Theorem~\ref{the:luzzaattoo}, we have
\[
	h_{\rm top}(F)
	= \sup_{\mu\in \cM_{\rm erg, <0}} h(\mu)
	= \sup_{\mu\in \cM_{\rm erg, >0}} h(\mu).
\]
\end{mtheorem}

%

Observe that the statement of Theorem~\ref{the:luzzaattoo} naturally extends to any invariant measure which is in the closed convex hull of $\cM_{\rm erg,0}$. However, in general there may exist invariant measures with zero  exponent that are not  in this hull and those would not necessarily be approximated by ergodic measures. The existence of such measures so far remains as an open question. If they do exist, then we provide some of their properties in Corollary~\ref{cor:main3twin}.
Compare Figure~\ref{fi.neu} for illustration.

\begin{figure}
\begin{minipage}[h]{\linewidth}
\centering
 \begin{overpic}[scale=.35]{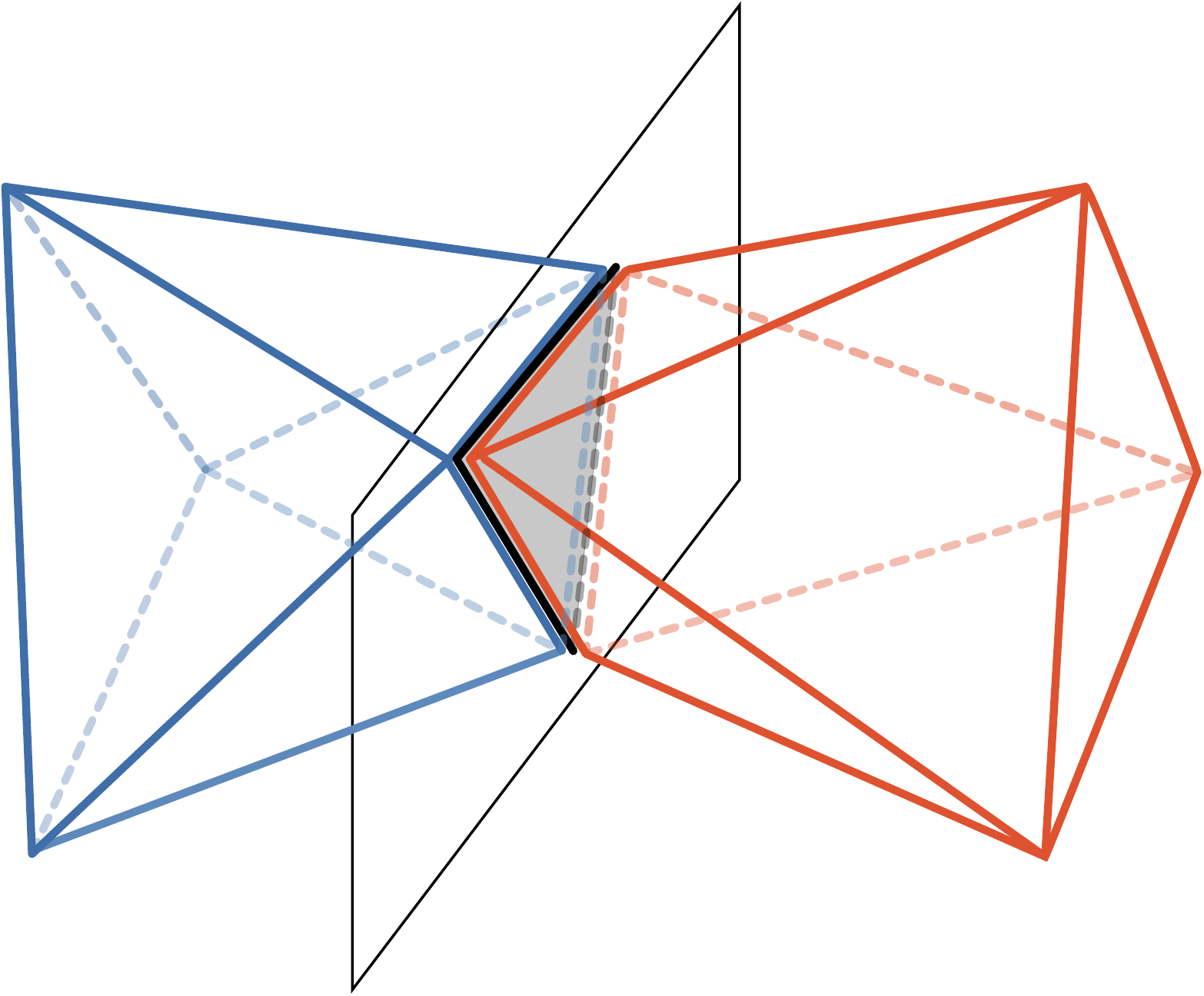}
  			\put(73,72){\small\textcolor{red}{$\clocon(\cM_{\rm erg,>0})$}}
  			\put(3,72){\small\textcolor{blue}{$\clocon(\cM_{\rm erg,<0})$}}
  			\put(32,12){\small\textcolor{black}{$\clocon(\cM_{\rm erg,0})$}}
  \end{overpic}
\caption{Schematic form of simplices of invariant measures}
\label{fi.neu}
\end{minipage}
\end{figure}	

Investigating finer properties of the measure space, we can quickly observe the following  
general ``twin principle'' (the simple proof is given in Section~\ref{sec:twins}).

\begin{mfact}[Twin measures]\label{pro:facttwinn}
	Consider a transitive step skew-product map $F$ as in~\eqref{eq:sp} whose fiber maps are $C^1$.
	
	Then for every $\mu\in\cM$ with $\chi(\mu)<0$ there exists $\widetilde\mu\in\cM$ satisfying $\chi(\widetilde\mu)\ge0$ and $h(\widetilde\mu)=h(\mu)$. If $\mu$ was ergodic, then $\widetilde\mu$ can be chosen ergodic.
\end{mfact}


Note that the construction in the proof of the above fact does not provide any information about the value of the exponent of the \emph{twin measure} $\widetilde\mu$.
One might be tempted to prove a ``perfect twin'' in the sense that to each hyperbolic ergodic positive entropy measure there is an ergodic measure with  equal entropy and negative fiber Lyapunov exponent.
As a first attempt, we can establish the following relation. We can ``push'' entropy of negative exponent measure to ``the other side'', though some amount of entropy and exponent is lost in the construction and this amount may increase the further away we are from zero exponent measures.

For the next -- more quantitative -- result we need to be a bit more precise. Assuming that $F$ is transitive and satisfies Axioms CEC$\pm$ and Acc$\pm$, by
Lemma~\ref{lem:commonintJ} below every closed sufficiently small  interval is a blending interval $J$. In order to be slightly more precise let us announce one of the properties required for Axiom CEC$+(J)$: there are constants $K_2,K_4$ so that for every sufficiently small interval $H\subset\bS^1$ intersecting $J$  there exists a finite sequence $(\eta_0\ldots\eta_{\ell-1})$ for some positive integer $\ell\sim  K_2\,\lvert\log\,\lvert H\rvert\rvert$ such that
\[
	f_{\eta_{\ell-1}}\circ \cdots \circ f_{\eta_0} (H)=
	f_{[\eta_0\ldots\,\eta_{\ell-1}]}(H)\supset B(J,K_4),
\]	
where $B(J,K_4)$ denotes the $K_4$-neighborhood of $J$.
Given a blending interval $J\subset\bS^1$, let $K_2(J)$ be the smallest number having this property for the interval $J$. Define
\[
	\Kdois(F)
	\eqdef \inf\{\Kdois(J)\colon J\text{ blending interval}\}.
\]
This number is intimately related with the inverse Lyapunov exponents. However, in general it might be much bigger.
One task, in particular in view of the estimates in Theorem~\ref{theo:main3twin}, is to minimize this number. We conjecture that in some important cases this number is equal to the inverse of the maximal fiber Lyapunov exponent, namely,
\begin{equation}\label{eq:Kdois}
	\Kdois(F)\eqdef
	\frac{1}{\overline\chi(F)},
	\quad\text{ where }\quad
	\overline\chi(F)
	\eqdef\max\{\chi(\mu)\colon\mu\in\cM_{\rm erg}\}.
\end{equation}
At this point we can only get the following
 natural lower bound:
 $$
 \Kdois(F)^{-1}\ge \log \,\lVert F\rVert
 $$
where
\begin{equation}\label{def:cunifconst}
		\D
		\eqdef \max_{i=0,\ldots,k-1}\,\max_{x\in\bS^1}
		\max\big\{\lvert f_i'(x)\rvert,\lvert(f^{-1}_i)'(x)\rvert\big\}.
	\end{equation}

\begin{mtheorem}\label{theo:main3twin}
Consider a transitive skew-product map $F$ as in~\eqref{eq:sp} whose fiber maps are $C^1$. Assume that $F$ satisfies Axioms CEC$\pm$ and Acc$\pm$.

	Then for every $\mu\in \cM_{\rm erg}$ with $\alpha= \chi(\mu) <0$  for every $\delta>0$ and $\gamma>0$, for every $\beta>0$,
	there is a compact $F$-invariant topologically transitive hyperbolic set $\widehat\Gamma$ such that
\begin{itemize}
\item[1.] its topological entropy satisfies
	\[
		h_{\rm top}(F,\widehat\Gamma)
		\ge \frac{h(\mu)}{1+\Kdois(F)(\beta+\lvert\alpha\rvert) } - \gamma,
	\]
\item[2.] for every $\nu\in \cM_{\rm erg}(\widehat\Gamma)$ we have
	\[
		 \frac \beta {1 + \Kdois(F)(\beta+\lvert\alpha\rvert)}  - \delta
		 < \chi(\nu) <
		 \frac \beta {1+\frac{1}{\log \D}(\beta+\lvert\alpha\rvert)}+ \delta,
	\]		
\item[3.] for every $\nu\in\cM(\widehat\Gamma)$ we have
	\[
		d(\nu,\mu)
		<\frac{\Kdois(F)(\beta+\lvert \alpha\rvert)}
			{1+\Kdois(F)(\beta+\lvert \alpha\rvert)}
		+\delta,
	\]	
	where $d$ is a metric which generates the weak$\ast$ topology.
\end{itemize}
The same conclusion is true for $\alpha > 0$ and every $\beta < 0$, changing in the assertion $\beta+|\alpha|$ to $|\beta|+\alpha$.

If $h(\mu)=0$ then $\widehat\Gamma$ is a hyperbolic periodic orbit.
\end{mtheorem}

As for Theorem~\ref{the:luzzaattoo}, we will prove Theorem~\ref{theo:main3twin} only in the case $h(\mu)>0$. 

Theorems~\ref{the:luzzaattoo} and~\ref{theo:main3twin} have the following immediate ``twin property'' 
corollary for nonhyperbolic measures.

\begin{mcorollary}\label{cor:main3twin}
Under the assumptions of Theorem~\ref{the:luzzaattoo},
for every  $\mu\in\cM$ with $\chi(\mu)=0$ for which there is a  sequence $(\mu^-_n)_n\subset\cM_{\rm erg,<0}$ which converges to $\mu$ in the weak$\ast$ topology 
there is
also a sequence $(\mu^+_n)_n\subset\cM_{\rm erg,>0}$ which converges to $\mu$ in the weak$\ast$ topology and satisfies $\lim_nh(\mu^+_n)=\lim_nh(\mu^-_n)$.
\end{mcorollary}

Let us now describe the organization and the essential ingredients of this paper. First, we state and investigate the above mentioned set of axioms, see Section~\ref{sec:axioms}, which are completely justified, see the examples and discussion in Section~\ref{sec:examples}. 
To deal with nonhyperbolic measures, we will require some very general distortion results which are give in Section~\ref{sec:generaltoos}. Our constructions are essentially based on so-called \emph{skeletons} for the dynamics which are orbit pieces that on one hand approximate well entropy, Lyapunov exponents, and measures and, on the other hand, are connected with a given reference blending interval provided by the axioms, see Section~\ref{sec:skeletons}.  Such skeletons allow us to construct hyperbolic sets ``around them'', for this we introduce the so-called \emph{multi-variable-time horseshoes} generalizing an idea in~\cite{LuzSan:13}, this
will be done in Section~\ref{sec:mulvarhor}. Thereafter in Section~\ref{sec:proofmainprops} we will construct explicit multi-variable-time horseshoes in our setting. Finally, 
Theorems~\ref{the:luzzaattoo},~~\ref{theocor:varprinc}, and~\ref{theo:main3twin}, and
Fact~\ref{pro:facttwinn}  are proved in Section~\ref{sec:mainproofs}.

\section{nonhyperbolic setting}\label{sec:axioms}

\subsection{Standing notation}

We equip the shift space $\Sigma_k$ with the standard metric $d_1(\xi,\eta)=2^{-n(\xi,\eta)}$, where $n(\xi,\eta)=\sup\{\lvert \ell\rvert\colon \xi_{ i}=\eta_{i}\text{ for }i=-\ell,\ldots,\ell\}$. We equip $\Sigma_k\times\bS^1$   with the metric $d((\xi,x),(\eta,y))=\sup\{d_1(\xi,\eta),\lvert x-y\rvert\}$.

The step skew-product structure of $F$ allows us to reduce the study of its dynamics to the study of the IFS generated by the family of maps $\{f_i\}_{i=0}^{k-1}$.
We use the following notations.
Every sequence $\xi=(\ldots\xi_{-1}.\xi_0\xi_1\ldots)\in\Sigma_k$ is given by $\xi=\xi^-.\xi^+$, where $\xi^+\in\Sigma_k^+\eqdef\{0,\ldots,k-1\}^{\bN_0}$ and $\xi^-\in\Sigma_k^-\eqdef\{0,\ldots,k-1\}^{-\bN}$.
Given  \emph{finite} sequences $(\xi_0\ldots \xi_n)$ and  $(\xi_{-m}\ldots\xi_{-1})$, we let
\[\begin{split}
    f_{[\xi_0\ldots\,\xi_n]}
    &\eqdef f_{\xi_n} \circ \cdots \circ f_{\xi_1}\circ f_{\xi_0}
    \quad\text{ and }\quad\\
    f_{[\xi_{-m}\ldots\,\xi_{-1}.]}
    &\eqdef  (f_{\xi_{-1}}\circ\ldots\circ f_{\xi_{-m}})^{-1}
    =(f_{[\xi_{-m}\ldots\,\xi_{-1}]})^{-1}.
\end{split}\]
For $n\ge0$, for notational convenience, we sometimes also write
\[
	f^n_\xi\eqdef f_{[\xi_0\ldots\,\xi_{n-1}]}
	\quad\text{ and }\quad
	f^{-m}_\xi\eqdef f_{[\xi_{-m}\ldots\,\xi_{-1}.]}.
\]	

We will study (fiber) Lyapunov exponents of $F$. They correspond to the Lyapunov exponents of the associated IFS defined as follows: given $X=(\xi,x)\in \Sigma_k\times\bS^1$ let
\[
	\chi(X)
	\eqdef
	 \lim_{n\to\pm\infty}\frac{1}{ n}\log\,\lvert (f^n_\xi)'(x)\lvert
\]
and in this definition we assume that both limits $n\to\infty$ and $n\to-\infty$ exist and that they are equal. Note that in our context the Lyapunov exponent is nothing but the Birkhoff average of a continuous function.

\subsection{Axioms}

Consider  fiber maps $f_0,\ldots,f_{k-1}\colon\bS^1\to\bS^1$
and its associated skew-product map $F$ defined as in~\eqref{eq:sp}.
We now introduce the properties satisfied by the associated IFS $\{f_i\}$.

Given a point $x\in\bS^1$, define its \emph{forward} and 
\emph{backward orbits} under the IFS  by
\[
	\mathcal{O}^+(x)	
	\eqdef \bigcup_{n\ge 0}
	\bigcup_{(\theta_0\ldots\theta_{n-1})}f_{[\theta_0\ldots\,\theta_{n-1}]}(x)
	\,\,\mbox{and} \,\,
	\mathcal{O}^{-}(x) \eqdef  \bigcup_{m\le 1}
	\bigcup_{(\theta_{-m}\ldots\theta_{-1})}f_{[\theta_{-m}\ldots\,\theta_{-1}.]}(x),
\]
respectively. Consider also the \emph{full orbit}
$$
	\mathcal{O}(x) \eqdef \mathcal{O}^+(x) \cup \mathcal{O}^-(x).
$$
Similarly, we define the orbits $\mathcal{O}^+(J),\mathcal{O}^-(J),\mathcal{O}(J)$ for any subset $J\subset\bS^1$.

The first axiom is very natural and is the corner stone of our constructions.

\medskip\noindent\textbf{Axiom T (Transitivity).}
There is a point  $x\in\bS^1$ such that the sets
 $\mathcal{O}^+(x)$ and
 $\mathcal{O}^-(x)$ are both dense in $\bS^1$.

\medskip
The next two axioms refer to the existence of intervals where appropriate compositions of the IFS $\{f_i\}$ have expanding and
contracting behavior.

\medskip\noindent\textbf{Axiom CEC+($J$) (Controlled Expanding forward Covering relative to a closed interval $J\subset\bS^1$).}
We say that the IFS $\{f_i\}$ satisfies CEC+($J$) if there exist positive constants $K_1,\ldots,K_5$ such that for every interval $H\subset\bS^1$ intersecting $J$ and satisfying $\lvert H\rvert<K_1$ we have
\begin{itemize}
\item  (controlled covering)  there exists a finite sequence $(\eta_0\ldots\eta_{\ell-1})$ for some positive integer $\ell\le  K_2\,\lvert\log\,\lvert H\rvert\rvert +K_3$ such that
\[
	f_{[\eta_0\ldots\,\eta_{\ell-1}]}(H)\supset B(J,K_4),
\]	
\item  (controlled expansion) for every $x\in H$ we have
\[
	\log \,\lvert  (f_{[\eta_0\ldots\,\eta_{\ell-1}]})'(x)\rvert\ge \ell K_5.
\]	
\end{itemize}

\medskip\noindent\textbf{Axiom CEC$-(J$) (Controlled Expanding backward Covering
relative to a closed interval $J\subset\bS^1$).} We say that the IFS $\{f_i\}$ satisfies CEC$-(J)$
if the IFS $\{f_i^{-1}\}$ satisfies the Axiom CEC$+(J$).
\smallskip


We observe that although Axioms CEC$\pm$($J$) do not provide an explicit lower bound for $\ell$, such a bound is obtained in Lemma~\ref{lem:warsaw}  at the end of this section.

\medskip

Finally, the last two axioms refer to covering properties of the IFS $\{f_i\}$.  Note that Axiom T implies immediately that for any nontrivial interval  $J\subset\bS^1$ one has that the sets $\mathcal{O}^+(J)$ and $\mathcal{O}^-(J)$ are both dense. We require a slightly stronger property.

\medskip\noindent\textbf{Axiom Acc$+$($J$) (Forward Accessibility relative to a closed interval $J$).}
We say that the IFS $\{f_i\}$ satisfies Acc$+(J$) if $\mathcal{O}^+(\interior J)
	=\bS^1$.

\medskip\noindent\textbf{Axiom Acc$-$($J$) (Backward Accessibility relative to a closed interval $J$).}
We say that the IFS $\{f_i\}$ satisfies Acc$-(J$) if $\mathcal{O}^-(\interior J)
	=\bS^1$.
\smallskip


We note that the IFS we consider has invertible fiber maps and hence has a naturally associated IFS generated by these inverse maps. By this correspondence, Axiom CEC$\pm(J)$ turns into CEC$\mp(J)$ and Axiom Acc$\pm(J)$ turns into Acc$\mp(J)$, respectively.

In the remainder of this paper we will mostly switch back to the point of view of the associated step skew-product $F$ and will say that $F$ satisfies the above axioms if the IFS does.
It follows from a standard genericity argument that if $F$ is transitive then there is a residual subset of $\Sigma_k\times \bS^1$ consisting of points having simultaneously forward and backward dense orbits. Having this in mind,
Axiom T is nothing but  transitivity of $F$.

We close this section with some simple consequences of the axioms above that we will use throughout the paper. The next remark follows straightforwardly from the definition
of the Axioms CEC$\pm$ and Acc$\pm$.

\begin{remark}
Assume that Axiom CEC$+(J)$ holds for some interval $J$. Then for any closed subinterval $I$ of $J$ Axiom CEC$+(I)$ holds with the same constants. The same assertion holds for Axiom CEC$-(J)$.
	
Assume Axiom CEC$+$($J$) and  Axiom Acc(+$J$) hold for some interval $J$. Then for any subinterval $I$ of $J$ Axiom Acc$+(I$) also holds. The same assertion holds for Axiom Acc$-(J$) with Axiom CEC$-(J)$.
\end{remark}

We state first an immediate consequence of compactness of $\bS^1$.

\begin{lemma}\label{lem:idiota}
	Assume that there is a closed interval $J\subset\bS^1$ such that the IFS $\{f_i\}$ satisfies Acc$+(J)$ (satisfies Acc$-(J)$).
	
	Then there exists a number $m_{\rm f}\ge1$ (a number $m_{\rm b}\ge1$) depending only on $J$ such that for every $x\in\bS^1$ there is a finite sequence $(\theta_1\ldots\theta_r)$, $r\le m_{\rm f}$, (a finite sequence $(\beta_1\ldots\beta_s)$, $s\le m_{\rm b}$) depending  on $x$, such that
	\[
		f_{[\theta_1\ldots\,\theta_r.]}(x)\in J
		\quad \quad
		\big(\text{such that }f_{[\beta_1\ldots\,\beta_s]}(x)\in J\big).
	\]
\end{lemma}

\begin{lemma}[Transitivity gives a common interval]\label{lem:commonintJ}
Assume that the IFS $\{f_i\}$ satisfies Axiom T.
	Assume that there are  closed intervals $J^+$ and $J^-$ such that
	the IFS $\{f_i\}$ satisfies CEC$+(J^+)$, Acc$+(J^+)$,
	CEC$-(J^-)$, and Acc$-(J^-)$.

	Then there are positive constants $K_1,\dots, K_5$, and $K_6>0$ such that for every $x\in\bS^1$, 
	for every $\delta<K_6$ the interval $J=\overline{B(x,\delta)}$ satisfies Axioms  CEC$+(J)$, Acc$+(J)$,
	CEC$-(J)$, and Acc$-(J)$ with these constants.

\end{lemma}

\begin{proof}
Assume that CEC$+(J^+)$ holds with constants $K_1,\ldots,K_5$.
By compactness of $\bS^1$ and  Axiom Acc$+(J^+)$, the circle $\bS^1$ is covered by a finite union of open sets which are images of $\interior J^+$, 
$$
\bS^1\subset\bigcup_{i=1}^m f_{[\theta_1^i\ldots\,\theta_{r_i}^i]}(\interior J^+).
$$ 
Let $\bar r\eqdef\max_{i=1,\dots, m} r_i$.

Let $J'$ be the concentric
interval contained in $J^+$ of length $\lvert J^+\rvert/2$
(that is, the distance of each point of  the boundary of $J'$ to the boundary of $J^+$ is
$\lvert J^+\rvert/4$).
By Axiom T,  for each $x\in \mathbb{S}^1$ there are
$s(x)\ge1$ and
a finite sequence $(\beta_1\ldots\beta_{s(x)})$ such that $f_{[\beta_1\ldots\,\beta_{s(x)}]}(x)\in J'$.
We also fix   $\delta(x)>0$ sufficiently small such that
$f_{[\beta_1\ldots\,\beta_{s(x)}]}(B(x,4\delta (x)))\subset J^+$.
An argument of compactness provides $\delta^+>0$ and $\bar s\ge1$ such that for every
$x\in \mathbb{S}^1$ there is a sequence
$(\beta_1\ldots\beta_{s})$, $s\le \bar s$, such that
$f_{[\beta_1\ldots\,\beta_{s}]}(B(x,2\delta^+))\subset J^+$.

Take any interval  $H^+\subset f_{[\beta_1\ldots\,\beta_s]}(B(x,2\delta^+))$ with
$\lvert H^+\rvert<K_1$. By CEC$+(J^+)$ applied to $H^+$, there exists a finite sequence $(\eta_0\ldots\eta_{\ell-1})$ with 
$$
\ell\le K_2\lvert\log\,\lvert H^+\rvert\rvert+K_3
$$ 
such that 
$$
f_{[\eta_0\ldots\,\eta_{\ell-1}]}(H^+)\supset J^+
\quad
\mbox{and}
\quad \log\,\lvert (f_{[\eta_0\ldots\,\eta_{\ell-1}]})'(y)\rvert \ge \ell K_5
$$ 
for every $y\in H^+$.

Let now $J\eqdef B(x,\delta^+)$.  For every interval $H\subset\bS^1$ intersecting $J$ with $\lvert H\rvert<\delta^+$ one has $H\subset B(x,2\delta^+)$, thus by the above choices,
there is some $i$ such that
\[
	f_{[\bar\eta_0\ldots\,\bar\eta_{j-1}]}(H) \supset B(x,\delta),
	\quad\text{ where }\quad
	(\bar\eta_0\ldots\,\bar\eta_{j-1})
	\eqdef (\beta_1\ldots\,\beta_s\eta_0\ldots\eta_{\ell-1}\theta_1^i\ldots\,\theta_{r_i}^i)
	.
\]
Since $j=s+\ell+r_i$ where
$s\le \bar s$ and $r_i\le \bar r$ we get the announced covering and expanding properties after
replacing the constants.

We repeat the previous construction with properties CEC$-(J^-)$ and Acc$-(J^-)$ obtaining
a number $\delta^-$. Now it is enough to take $\delta=\min\{\delta^+,\delta^-\}$.
\end{proof}

%
%
%
%

The next two lemmas follow straightforwardly from the definitions and their proofs are omitted.

\begin{lemma}[A common interval gives transitivity and density of periodic points]
	Assume that  there is a closed interval $J$ such that the IFS $\{f_i\}$ satisfies Axioms  CEC$+(J)$, Acc$+(J)$, CEC$-(J)$, and Acc$-(J)$.
	
	Then the IFS satisfies Axiom T.
Moreover, $\Sigma_k\times\bS^1$ is the closure of periodic orbits with negative/positive fiber exponents.
\end{lemma}

\begin{lemma}\label{lem:transssitivo}
	Assume that there are intervals $J^+$ and $J^-$ such that the IFS satisfies Axioms CEC$+(J^+)$ and Acc$\pm(J^+)$ and Axioms CEC$-(J^-)$ and Acc$\pm(J^-)$.  Suppose that every $x\in\bS^1$ has a forward and a backward iterate in the interior of $J^+$ and has a forward and a backward iterate in  the interior of $J^-$.

Then, the IFS satisfies Axiom T and there is an interval $J$ such that the IFS satisfies Axioms CEC$\pm(J)$ and Acc$\pm(J)$.
\end{lemma}

Having in mind the previous results we introduce the following definition.

\begin{definition}
	We say that a step skew-product map $F$ as in~\eqref{eq:sp} satisfies Axioms CEC$\pm$ and Acc$\pm$ if there is some closed interval $J\subset\bS^1$ satisfying Axioms CEC$\pm(J)$ and Acc$\pm(J)$.
\end{definition}

 Axiom CEC$+(J)$ demands only  an upper bound for the
size $\ell$ of the covering sequence $(\eta_0\ldots \eta_{\ell-1})$, depending uniformly on the size of the interval $H$ intersecting $J$.
The next lemma claims that
 the size of  the covering sequence can be also bounded from below.
One can state an analogous statement for the Axiom CEC$-(J)$.

\begin{lemma}\label{lem:warsaw}
	Assume that the IFS $\{f_i\}$ satisfies Axiom CEC$+(J)$ with constants $K_1,\ldots,K_5$. Then for every interval $H$ intersecting $J$ and satisfying $\lvert H\rvert<K_1$, there is a subinterval $\widehat H\subset H$ and a constant $\iota=\iota(H)$ satisfying
\begin{itemize}
\item there is a finite sequence $(\rho_0\ldots \rho_{\iota-1})$,
\[
	K_2\,\lvert\log\,\lvert H\rvert\rvert+K_3
	\le \iota
	\le 2(K_2\,\lvert\log\,\lvert H\rvert\rvert+K_3),
\]	
such that
\[
	f_{[\rho_0\ldots\,\rho_{\iota-1}]}(\widehat H)\supset B(J,K_4),
\]	
\item for every $x\in \widehat H$ we have
\[
	\log \,\lvert  (f_{[\rho_0\ldots\,\rho_{\iota-1}]})'(x)\rvert\ge \iota K_5.
\]		
\end{itemize}
\end{lemma}

\begin{proof}
	The proof is by induction. Let $H_0=H$ and consider $(\eta^0_0\ldots\eta^0_{\ell_0-1})$ given by Axiom CEC$+(J)$ applied to $H_0$. Now for $j\ge0$ consider the following recursion:
\begin{itemize}
\item [i)]  if $\ell_0+\cdots+\ell_j< K_2\,\lvert\log\,\lvert H_j\rvert\rvert+K_3$ then we pick an interval $H_{j+1}\subset f_{[\eta^j_0\ldots\,\eta^j_{\ell_j-1}]}(H_j)$ satisfying $\lvert H_{j+1}\rvert=\lvert H_j\rvert=\lvert H_0\rvert$ and repeat the recursion;
\item [ii)] otherwise stop the recursion and let $\iota\eqdef \ell_0+\cdots+\ell_j$.
\end{itemize}
Clearly, in the above recursion there is a first $j\ge0$ such that case ii) applies and therefore, by construction and the fact that $\ell_j\le K_2\lvert\log\lvert H\rvert\rvert+K_3$, we have
\[
	K_2\,\lvert\log\,\lvert H\rvert\rvert+K_3
	\le \iota 
	\le 2K_2\,\lvert\log\,\lvert H\rvert\rvert+2K_3
\] 	
and we put
\[
	(\rho_0\ldots\rho_{\iota-1})
	= (\eta^0_0\ldots\eta^0_{\ell_0-1}\ldots\eta^j_0\ldots\eta^j_{\ell_j-1}).
\]
We pick the subinterval $\widehat H = (f_{[\rho_0\ldots\,\rho_{\iota-1}]})^{-1}\big(B(J,K_4)\big)\subset H$.
By construction, $\widehat H$ satisfies the covering property. To get the expansion just note that
\[
	\log \,\lvert  (f_{[\rho_0\ldots\,\rho_{\iota-1}]})'(x)\rvert
	\ge (\ell_0+\ldots+\ell_j)K_5=\iota K_5.
\]
This proves the lemma.
\end{proof}

\section{Some general tools}\label{sec:generaltoos}

In this section, we continue to consider a step skew-product map $F$ as in~\eqref{eq:sp} with $C^1$ fiber maps. We derive a number of ``uniformization'' results for ergodic measures following Littlewood's heuristic principles (here using the fact that due to Egorov's theorem every pointwise converging sequence of measurable functions is nearly uniformly convergent). We also state some very general distortion results which, in particular, allow us to deal with zero exponent orbits.

\subsection{Approximation of positive entropy ergodic measures}

The following statement is a consequence of ergodicity,
the definition of a Lyapunov exponent, the Brin-Katok theorem, the Birkhoff ergodic theorem, and the Egorov theorem. Recall the definition of separated points, see~\cite[Chapter 7]{Wal:82}.

\begin{proposition}\label{pro:BriKat}
Consider a skew-product map
 $F$ as in \eqref{eq:sp} whose fiber maps are $C^1$.
Let $\mu\in\cM_{\rm erg}$ be a measure satisfying $h(\mu)>0$. Let $\alpha=\chi(\mu)$.
Consider continuous functions $\varphi_1,\ldots,\varphi_\ell\colon\Sigma_k\times\bS^1\to\bR$ and put $\overline\varphi_j=\int\varphi_j\,d\mu$.
Let  $A\subset\Sigma_k\times\bS^1$ be a measurable set with $\mu(A)>0$.

Given $\kappa\in(0,\mu(A)/4)$, $r\in(0,1)$, and $\varepsilon_H\in(0,1)$, for every $\varepsilon_E>0$ small enough there exist $n_0=n_0(\varkappa,\varepsilon_H)\ge1$ and a set $\Lambda'\subset\Sigma_k\times\bS^1$ satisfying $\mu(\Lambda')>1-\varkappa$ such that:
\begin{itemize}
\item[(1)] there exists $K_0=K_0(\kappa,\varepsilon_E)>1$ such that for every $n\ge0$ and every $X=(\xi,x)\in\Lambda'$ we have
\[
	K_0^{-1}e^{n(\alpha-\varepsilon_E)}
	\le \lVert (f^n_\xi)'(x)\rVert
	\le K_0 e^{n(\alpha+\varepsilon_E)},
\]
and for every $j=1,\ldots,\ell$ we have
\[
	-K_0+n(\overline\varphi_j-\varepsilon_E)
	\le \sum_{\ell=0}^{n-1}\varphi_j(F^\ell(X))
	\le K_0+n(\overline\varphi_j+\varepsilon_E),
\]
\item[(2)] for every $n\ge n_0$ there is $m\in\{n,\ldots,n+\lceil r n\rceil+1\}$ and a set  of $(m,1)$-separated points $\{X_i\}\subset A\cap\Lambda'$ of cardinality $M_m(A)$ satisfying
\[
	M_m(A)
	\ge \big(\mu(A)-\kappa\big) \cdot e^{m(h(\mu)-\varepsilon_H)}
\]
and
\[
	F^m(X_i)\in A\,.
\]
\end{itemize}
\end{proposition}


Before proving the proposition we make some preliminary remarks.
Given a positive integer $n$ and a positive number $\varrho$, for a point $X\in\Sigma_k\times\bS^1$ we consider the \emph{$(n,\varrho)$-Bowen ball} centered at $X$
\[
	B_n(X,\varrho)\eqdef\bigcap_{\ell=0}^{n-1}F^{-\ell}\big(B\big(F^\ell(X),\varrho\big)\big),
\]	
where $B(Y,\varrho)$ denotes the open ball of radius $\varrho$ centered at $Y$. We will also consider the analogously defined $(n,1)$-Bowen ball relative to the base dynamics $\sigma\colon\Sigma_k\to\Sigma_k$ and recall that for given $\xi\in\Sigma_k$ it is simply the $n$th level cylinder
\[
	B_n(\xi,1)
	\eqdef [\xi_0\ldots\xi_{n-1}]
	= \{\eta\in\Sigma_k\colon \eta_i=\xi_i\text{ for }i=0,\ldots,n-1\}.
\]	
We also note that any pair of disjoint level $n$ cylinders gives rise to $(n,1)$-separated sequences in $\Sigma_k$ and hence to $(n,1)$-separated points in $\Sigma_k\times\bS^1$, we will use this fact a couple of times.

Consider the natural projection $\varpi\colon\Sigma_k\times\bS^1\to\Sigma_k\colon (\xi,x)\mapsto \xi$ to the first coordinate and observe that the pushforward measure $\nu\eqdef\varpi_\ast\mu$ is ergodic invariant with respect to $\sigma\colon\Sigma_k\to\Sigma_k$. Note that $h(\nu)=h(\mu)$ (this follows from~\eqref{eq:entropyproj}).

\begin{proof}[Proof of Proposition~\ref{pro:BriKat}]
Fix $\kappa,r,\varepsilon_H>0$ as in the hypotheses.
We start with a preliminary estimate. Given $\mu$ as in the proposition let 
$\nu\eqdef\varpi_\ast\mu$. 

\begin{lemma}
	There is a set $\Lambda_1\subset\Sigma_k\times\bS^1$ of measure at least $1-\kappa/4$ and a number $n_1=n_1(\kappa,\varepsilon_H)\ge1$ such that for every $m \ge n_1$
	and every $X=(\xi,x)\in\Lambda_1$ we have
\begin{equation*}
	e^{- m (h(\mu)+\varepsilon_H/2)}
	\le \mu(B_{m}(X,1))
	\le e^{-m (h(\mu)-\varepsilon_H/2)}
\end{equation*}
and
\begin{equation}\label{eq:formullanumeasure}
	e^{-m (h(\mu)+\varepsilon_H/2)}
	\le \nu([\xi_0\ldots\xi_{m-1}])
	\le e^{-m (h(\mu)-\varepsilon_H/2)}.
\end{equation}
\end{lemma}

\begin{proof}
By the Brin-Katok theorem~\cite{BriKat:83}, there is a set $\Lambda\subset\Sigma_k\times\bS^1$ with $\mu(\Lambda)=1$ so that every $X\in\Lambda$ satisfies
\[
	\lim_{\varrho\to0}\limsup_{n\to\infty}-\frac1n\log\mu(B_n(X,\varrho))
	=\lim_{\varrho\to0}\liminf_{n\to\infty}-\frac1n\log\mu(B_n(X,\varrho))
	=h(\mu).
\]
Analogously, for $\nu$-almost every $\xi\in\varpi(\Lambda)$
\[
	\limsup_{n\to\infty}-\frac1n\log\nu([\xi_0\ldots\xi_{n-1}])
	=\liminf_{n\to\infty}-\frac1n\log\nu([\xi_0\ldots\xi_{n-1}])
	=h(\nu)=h(\mu).
\]
Now apply the Egorov theorem.
\end{proof}

We now prove (1) in the proposition. We will only derive the conclusion for the  Lyapunov exponents, the one for the potentials $\varphi_j$ is completely analogous.
By ergodicity, for $\mu$-almost every $X=(\xi,x)$ we have
\[
	\lim_{n\to\infty}\frac1n\log\,\lvert (f^n_\xi)'(x)\rvert = \alpha.
\]
By the Egorov theorem, there is a set $\Lambda_2\subset\Sigma_k\times\bS^1$ of $\mu$-measure at least $1-\kappa/4$ and a number $n_2=n_2(\kappa,\varepsilon_E)$
 such that for every $X=(\xi,x)\in\Lambda_2$ and every $m \ge n_2$ we have
\begin{equation}\label{eq:asser}
	e^{m (\alpha-\varepsilon_E)}
	\le \lvert (f^m_\xi)'(x)\rvert
	\le e^{m (\alpha+\varepsilon_E)}.
\end{equation}
Let
\begin{multline*}
	K_0\eqdef\max_{n=0,\ldots,n_2-1}
	\max_{(\xi_0\dots \xi_{n-1})}
	\max_{x\in\bS^1}\\
		\Big\{\lvert (f_{[\xi_0\dots \xi_{n-1}]} )'(x)\rvert\,e^{-n(\alpha+\varepsilon_E)},
			\lvert (f_{[\xi_0\dots \xi_{n-1}]})'(x)\rvert^{-1}e^{n(\alpha-\varepsilon_E)}\Big\}.
\end{multline*}
With this choice, for every $X=(\xi,x)\in\Lambda_2$ for every $n=0,\ldots,n_2-1$ we have the assertion of (1) while for every $n\ge n_2$ we have~\eqref{eq:asser} proving (1).

To show (2), let now
\begin{equation}\label{defCCC}
	C\eqdef\min\Big\{\kappa,r,\frac{\mu(A)}{4}\Big\}
	\in(0,1).
\end{equation}

\begin{lemma}
	There exists a measurable set $\Lambda_3\subset\Sigma_k\times\bS^1$
	 of measure at least $1-\kappa/4$ and a number $n_3=n_3(\kappa,r,A)\ge 1$ such that for every $X\in\Lambda_3$ and $m \ge n_3$ we have
	\[
	\left\lvert
	\frac1m{\rm card}\big\{\ell\in\{0,\ldots,m-1\}\colon F^\ell(X)\in A\big\} -
	 \mu(A)\right\rvert \le C^2.
	\]
\end{lemma}

\begin{proof}
By the Birkhoff theorem, there is a full measure set $\Lambda$ such that  for every $X\in\Lambda$ we have
\[
	\lim_{n\to\infty}\frac{1}{n}{\rm card}\big\{\ell\in\{0,\ldots,n-1\}
		\colon F^\ell(X)\in A\big\}
	=\mu(A).
\]	
Now apply the Egorov theorem.
\end{proof}

We can assume that $n_3$ has been chosen large enough such that
\begin{equation}\label{eq:formula}
	n_3\,r(\mu(A)-3C)>1
\end{equation}	
and thus for every $X\in\Lambda_3$ and every $n\ge n_3$
\[\begin{split}
	\card&\{\ell\colon n\le \ell < n(1+r),\,\, F^\ell(X)\in A\} \\
	&\ge
 	n(1+r)\big(\mu(A)-C^2\big)
	-(n-1)\mu(A)-(n-1)C^2\\
	&=
 	(nr+1)\mu(A)-(2n+nr-1)C^2\\
	&\ge
 	nr\big(\mu(A)-C^2\big)-2nC^2\\
\text{(with~\eqref{defCCC})}\,\,\,	
	&>
 	nr\big(\mu(A)-C\big)-2nr C\\
	&= nr(\mu(A)-3C) >1,
\end{split}\]	
where the last inequality follows from~\eqref{eq:formula}.

Now let $n_0=\max\{n_1,n_2,n_3\}$ and $\Lambda=\Lambda_1\cap\Lambda_2\cap\Lambda_3$. Assume also that for every $n\ge n_0$ we have
\begin{equation}\label{condition}
	n<e^{n\varepsilon_H/2}.
\end{equation}
Observe that $\mu(\Lambda_i)>1-\kappa/4$ for $i=1,2,3$ implies $\mu(\Lambda)>1-\kappa$.

Observe that the set $A\cap\Lambda$ consists of  points having orbits which start  and end in $A$, however which need possibly different number of iterations for that (between $n$ and $n+\lceil r n\rceil$).
We will now consider the separated subsets with \emph{equal} return time and 
select a subset with maximal cardinality having this property. 
In this way, the cardinality of the selected set is still comparable with entropy.
For each $\ell$ with $n\le \ell< n+\lceil r n\rceil+1$ let
\[
R_\ell\eqdef \left\{X_i\in \Lambda \colon F^\ell(X_i)\in A\right\}
\]
be the set of points which have the same time $\ell$ of return to $A$.

In order to obtain a large separated set of points with \emph{equal} return times we do the following construction. 
Pick an index $m\in\{n,\ldots,n+\lceil r n\rceil+1\}$  satisfying
\[
	\card R_m =\max_{n\le \ell<n+\lceil r n\rceil+1}\card R_\ell.
\]
Let
\[
	A'\eqdef A\cap \Lambda
	\quad\text{ and }\quad
	S'\eqdef \varpi(A').
\]	
and observe that $\nu(S')\ge\mu(A')\ge\mu(A)-\kappa>0$.
Choose any point $X_1=(\xi^1,x_1)\in A'$.
Let $S_1=S'\setminus [\xi^1_0\ldots\xi^1_{m-1}]$.
We continue inductively: choose any $\xi^\ell\in S_{\ell-1}$, let $S_\ell=S_{\ell-1}\setminus [\xi^\ell_0\ldots\xi^\ell_{m-1}]$, rinse and repeat. As by~\eqref{eq:formullanumeasure}
\[
	\nu(S_\ell) \geq \nu(S') - \ell e^{-m(h(\mu)-\varepsilon_H/2)},
\]
we can continue the procedure for at least $M$ steps, where
\begin{equation}\label{eq:Meq}
	M\ge \big\lceil\nu(S') \cdot e^{m(h(\mu)-\varepsilon_H/2)}\big\rceil
		\ge  \big(\mu(A)-\kappa\big) \cdot e^{m(h(\mu)-\varepsilon_H/2)}.
\end{equation}
By construction, the resulting set of sequences $\{\xi^1, \ldots,\xi^M\}\subset S'$
is $(m,1)$-separated set (with respect to $\sigma\colon\Sigma_k\to\Sigma_k$).
For every sequence $\xi^i$ there exists a point $X_i\in A'$ with $X_i=(\xi^i,x_i)$ for some $x_i\in\bS^1$. Note that the set 
$\{X_1,\ldots,X_M\}$ 
is $(m,1)$-separated (with respect to $F$).

With~\eqref{condition} we have $nr <n<e^{n\varepsilon_H/2}$ and hence with~\eqref{eq:Meq} we obtain
\[
	\card R_m
	\ge \frac{M}{nr}
	\ge \big(\mu(A)-\kappa\big)\cdot e^{m(h(\mu)-\varepsilon_H/2)}e^{-m\varepsilon_H/2}
	= \big(\mu(A)-\kappa\big)\cdot e^{m(h(\mu)-\varepsilon_H)}.
\]
This proves item (2) and completes the proof of the proposition.
\end{proof}

\subsection{Distortion}

We will need some auxiliary distortion results. They include, in particular, distortion in a neighborhood of orbits with zero fiber Lyapunov exponent.

Given a set $Z\subset\bS^1$ and a differentiable map $g$ on $Z$, we denote by
\[
	\dist g|_Z
	\eqdef \sup_{x,y\in Z}\frac{\lvert g'(x)\rvert}{\lvert g'(y)\rvert}
\]
the \emph{maximal distortion} of $g$ on $Z$.
Given $\delta>0$, we consider the \emph{modulus of continuity} of the function $\log\,\lvert g'\rvert$ defined by
\[
	\Mod (\log\,\lvert g'\rvert,\delta,x)
	\eqdef \max\big\{
		\big\lvert \log\,\lvert g'(y)\rvert - \log\,\lvert g'(x)\rvert\big\rvert
			\colon \lvert y-x\rvert\le \delta\big\}.
\]

Considering the IFS $\{f_i\}$, let
\begin{equation}\label{eq:modulos}
	\Mod(\delta)
	\eqdef \max_{i=0,\ldots,k-1}\,\max_{x\in\bS^1}\,\Mod (\log \, |f_i'|,\delta,x).
\end{equation}
Clearly, $\Mod(\delta)\to0$ as $\delta\to0$.

\begin{proposition}[Distortion]
	Consider a skew-product map
 $F$ as in \eqref{eq:sp} whose fiber maps are $C^1$.
	Given $\varepsilon_D>0$, choose $\delta_0>0$ such that $\Mod(2\delta_0)\le \varepsilon_D$.
	Assume that $(\xi,x)\in\Sigma_k\times\bS^1$ is such that 
	there are $r>0$ and $m\ge1$ such that for every $\ell=0,\ldots,m-1$ we have
\[
	\lvert (f_\xi^\ell)'(x)\rvert<\frac1r\delta_0e^{-\ell\varepsilon_D}.
\]	

Then for every $\ell\in\{0,\ldots,m\}$  we have
\[
	\big\lvert\log\dist f_\xi^\ell |_{[x-r,x+r]}\big\rvert \le \ell\varepsilon_D\,.
\]	
\end{proposition}

\begin{proof}
	Let $Z=B(x,r)$.
	The proof is by (finite) induction on $\ell$. Note that the claim holds for $\ell=0$. Suppose that the claim holds for $\ell=i$. This means that we have $\lvert\log \dist f_\xi^i|_Z\rvert\le i\varepsilon_D$, which by the hypothesis of the proposition implies that
\[
	\lvert f_\xi^i(Z)\rvert \le
	 \frac{1}{r} \delta_0 e^{-i\varepsilon_D}\cdot e^{i\varepsilon_D}
	 \cdot \lvert Z\rvert 
	= 
	2\delta_0.
\]	
Hence $\lvert\log\dist f|_{f_\xi^i(Z)}\rvert\le \varepsilon_D$. Now the chain rule implies $\lvert\log\dist f_\xi^{i+1}|_Z\rvert\le (i+1)\varepsilon_D$ which is the claim for $i+1$. This proves the proposition.
\end{proof}

In a similar manner the following result can be shown.

\begin{corollary}[Distortion for zero exponents]\label{cor:distortionC1}
	Consider a skew-product map
 $F$ as in \eqref{eq:sp} whose fiber maps  are $C^1$.
	Given $\varepsilon_D>0$, choose $\delta_0>0$ such that $\Mod(2\delta_0)\le \varepsilon_D$. Given $\varepsilon\in(0,1)$, $m\ge1$, $K_0>0$, and $(\xi,x)\in\Sigma_k\times\bS^1$ satisfying for all $\ell\in\{0,\ldots,m\}$
\[
	\lvert (f_\xi^\ell)'(x)\rvert \le K_0e^{\ell\varepsilon}.
\]	
	
Then with $Z= B(x,\delta_0 K_0^{-1}e^{-m(\varepsilon+\varepsilon_D)})$
 for every $\ell\in\{0,\ldots,m\}$ 
 we have
\[
	\left\lvert\log\dist f_\xi^\ell|_Z\right\rvert \le \ell\varepsilon_D\,.
\]	
\end{corollary}

We now provide one more distortion result. It is more specific to our step skew-product axiomatic setting, and not  in the general $C^1$ setting as above. We show that Axiom Acc$- (J$) allows us to strengthen Axiom CEC$+(J$) in the following way.

\begin{lemma}\label{lem:disaxicec}
Consider a skew-product map
 $F$ as in \eqref{eq:sp} whose fiber maps  are $C^1$.
Assume that there is a closed interval $J\subset\bS^1$ such that Axiom CEC$+(J$) is satisfied with constants $K_1,\ldots,$ $K_5$ and that Axiom Acc$-(J$) is satisfied.

Then for every $\varepsilon_D>0$ there exist positive constants $K_3'$ and $K_D$ such that for every interval $H\subset\bS^1$ intersecting $J$ and satisfying $\lvert H\rvert<K_1$ we have
\begin{itemize}
\item (controlled covering) there exists some finite sequence $(\xi_0 \ldots \xi_{\iota-1})$ for some positive integer $\iota \leq K_2 \lvert\log\, \lvert H\rvert\rvert + K_3'$ such that
\[
	f_{[\xi_0\ldots\, \xi_{\iota-1}]}(H) \supset B(J, K_4),
\]
\item (controlled distortion) we have
\[
	\log \dist f_{[\xi_0\ldots\,\xi_{\iota-1}]}|_ H
	\leq \lvert\log\, \lvert H\rvert\rvert \cdot \varepsilon_D + \log K_D.
\]
\end{itemize}
\end{lemma}

\begin{proof}
Recall the definition of modulus of continuity $\Mod (\cdot) $ in~\eqref{eq:modulos} and, given $\varepsilon_D$, fix $\delta>0$ so that
\[
	\delta<K_1
	\quad\text{ and }\quad
	\Mod(\delta)<\frac{\varepsilon_D}{K_2}.
\]

We fix an interval $H$ as in the hypothesis and consider the corresponding finite sequence $(\eta_0\ldots\eta_{\ell-1})$ provided by Axiom CEC$+(J)$. Let $t\in\{1,\ldots,\ell\}$ be the smallest integer  when
\[
	\lvert H'\rvert  > \delta,
	\quad\text{ where }\quad
	H' \eqdef f_{[\eta_0\ldots\,\eta_{t-1}]}(H).
\]
By Axiom CEC$+(J$) we have
\[
	t \leq \ell \leq K_2\, \lvert\log \,\lvert H\rvert\rvert + K_3.
\]	

Assuming Acc$-(J)$, by Lemma~\ref{lem:idiota}, there is a universal number $m_{\rm b}$ (depending only on $J$) such that there is a finite sequence $(\beta_1\ldots\beta_s)$, $s\le m_{\rm b}$, such that
\[
	H''
	\cap J \neq \emptyset,
	\quad\text{ where }\quad
	H''\eqdef f_{[\beta_1\ldots\,\beta_s]}(H').
\]
By the definition of the universal constant $\D$ in~\eqref{def:cunifconst}, we have
\[
	 |H''| \geq \D^{-m_{\rm b}} \cdot |H'| > \D^{-m_{\rm b}} \cdot \delta.
\]

Finally, we apply Axiom CEC$+(J)$ to the interval $H''$ we obtain a finite sequence $(\overline\eta_0\ldots\overline\eta_{r-1})$ with $r \leq K_2 \,\lvert\log \,\lvert H''\rvert\rvert + K_3$ for which
\[
	f_{[\overline\eta_0\ldots\,\overline\eta_{r-1}]}(H'') \supset J.
\]

Define now the finite sequence
\[
	(\xi_0\ldots\xi_{\iota-1})
	\eqdef (\eta_0\ldots\eta_{t-1}
		\beta_1\ldots\,\beta_s\overline\eta_0\ldots\,\overline\eta_{r-1}).
\]		
We have
\[\begin{split}
	\iota
	&= t+s+r\\
	&\leq (K_2 |\log |H|| + K_3) + m_{\rm b} +(K_2 \,\lvert\log \,\lvert H''\rvert\rvert + K_3)\\
	&\leq (K_2 |\log |H|| + K_3) + m_{\rm b}
		+ (K_2 m_{\rm b}\log  \D
			+ K_2 \lvert\log \delta\rvert + K_3) \\
	&= K_2 |\log |H|| + K_3',
\end{split}\]
where $K_3'=K_3'(\delta)$ is the sum of the above remaining constants.
This completes the proof of the first part of the lemma (controlled covering).

To get the control of the distortion 
note  that the previous estimate  shows  $s+r\le K_3'$.
Recalling again the choice of $\delta$ and $t$, we get
\[
	\log \dist f_{[\xi_0\ldots\,\xi_{\iota-1}]}|_H
	\leq t\, \frac{\varepsilon_D}{K_2} + (s+r)\log \, \lVert F \rVert
	\le \lvert\log \,\lvert H\rvert \rvert \cdot \varepsilon_D + \frac{K_3}{K_2}\varepsilon_D+
	(s+r)\log \, \lVert F \rVert.
\]
Letting $K_D=K_D(\delta)\eqdef  \frac{K_3}{K_2}\varepsilon_D+ (s+r)\log \, \lVert F \rVert$, this shows the lemma.
\end{proof}

\section{Skeletons}\label{sec:skeletons}

The systems we consider provide  ``skeletons" of the dynamics, that is, orbit pieces approximating dynamical properties such as entropy and fiber exponent. These skeletons will serve as building pieces to construct transitive hyperbolic sets which are, in a certain sense, dynamically and ergodically homogeneous. Here, these orbit pieces will approximate either certain invariant sets (Skeleton property) or certain invariant measures (Skeleton$\ast$ property), respectively.
Compare Figure~\ref{fig.iteratesd}. Throughout this section we continue to consider
a skew-product map $F$ as in \eqref{eq:sp}.

\begin{definition}[Skeleton property]
	Given an interval $J\subset\bS^1$ and numbers $h\ge0$ and $\alpha\ge0$, we say that $F$ has the \emph{Skeleton property relative to $J$, $h$, and $\alpha$} if there exist $m_{\rm b},m_{\rm f}\in\bN$ (\emph{connecting times}) such that for any $\varepsilon_H\in(0,h)$ and $\varepsilon_E>0$ there exist $K_0,L_0\ge1$, and $n_0\ge1$ such that for every $m\ge n_0$ there exists a finite set $\fX=\fX(h,\alpha,\varepsilon_H,\varepsilon_E,m)=\{X_i\}$ of points $X_i =(\xi^i,x_i)$ (\emph{Skeleton}) satisfying:
\begin{itemize}
\item[(i)] the set $\fX$ has cardinality
\[
	\card\fX
	\ge L_0^{-1}e^{m(h-\varepsilon_H)},
\]
\item[(ii)] the sequences $(\xi^i_0\ldots\xi^i_{m-1})$  are all different,
\item[(iii)] for every $n=0,\ldots,m$
\[
	K_0^{-1} e^{n(\alpha-\varepsilon_E)}
	\le \lvert (f_{[\xi^i_0\ldots\,\xi^i_{n-1}]})'(x_i)\rvert
	\le K_0 e^{n(\alpha+\varepsilon_E)}.
\]	
\end{itemize}
Moreover, there are sequences $(\theta^i_1\ldots\theta^i_{r_i})$, $r_i\le m_{\rm f}$, and $(\beta^i_1\ldots\beta^i_{s_i})$, $s_i\le m_{\rm b}$, and points $x_i'\in J$ such that  for every $i$ we have
\begin{itemize}
\item[(iv)] $f_{[\theta^i_1\ldots\,\theta^i_{r_i}]}(x_i')=x_i$,
\item[(v)] $f_{[\xi^i_0\ldots\,\xi^i_{m-1}\beta^i_1\ldots\,\beta^i_{s_i}]}(x_i)\in J$.
\end{itemize}
\end{definition}

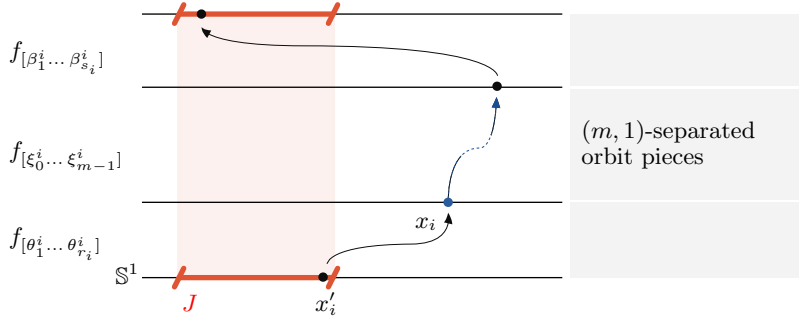
\begin{figure}
\begin{minipage}[h]{\linewidth}
\centering
 \begin{overpic}[scale=.50]{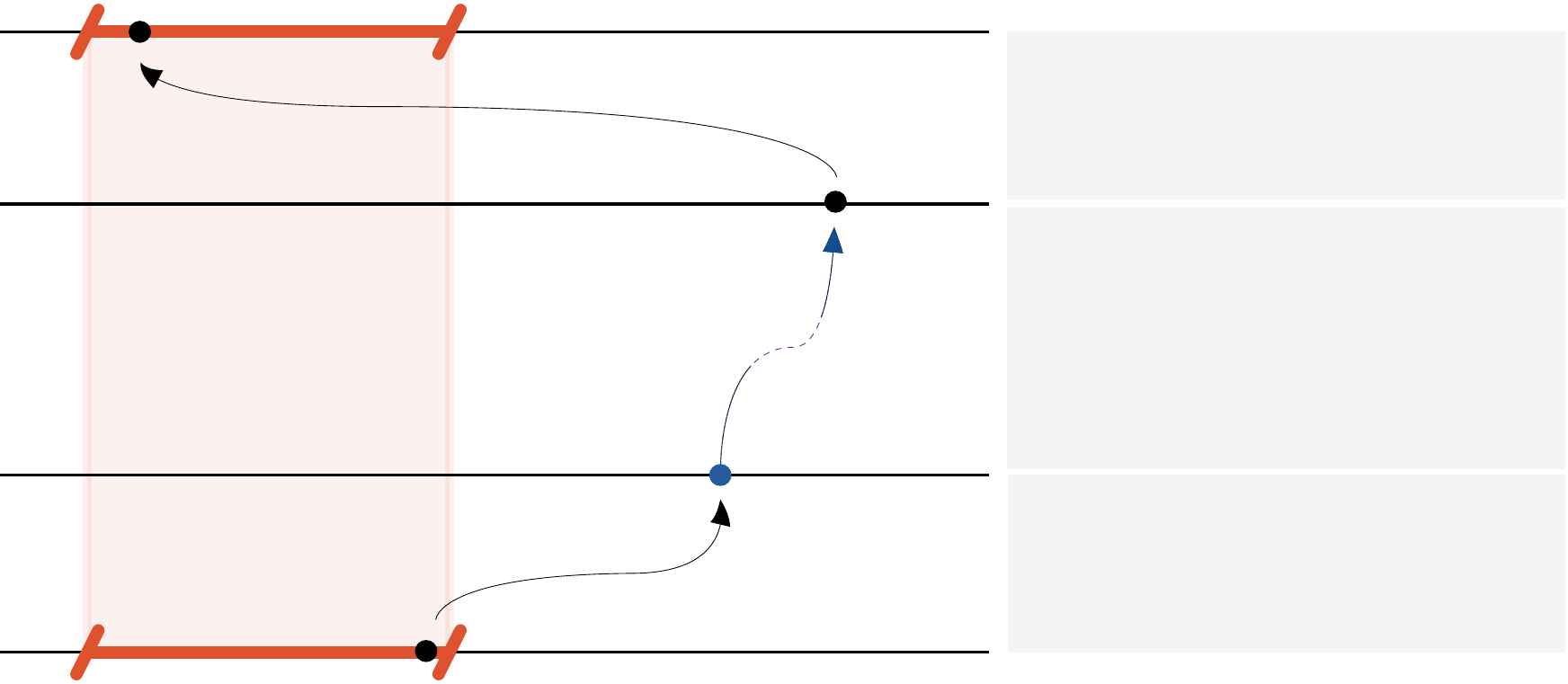}
 			\put(-20,7){\small$f_{[\theta^i_1\ldots\,\theta^i_{r_i}]}$}
 			\put(-20,20){\small$f_{[\xi^i_0\ldots\,\xi^i_{m-1}]}$}
 			\put(-20,35){\small$f_{[\beta^i_1\ldots\,\beta^i_{s_i}]}$}
  			\put(26,-3){\small$x_i'$}
  			\put(41,9.5){\small$x_i$}
  			\put(66,23){\small$(m,1)$-separated}
			\put(66,19){\small orbit pieces}
  			\put(-4,1){\small$\bS^1$}
  			\put(6,-3){\small\textcolor{red}{$J$}}
  \end{overpic}
\caption{Skeleton property}
\label{fig.iteratesd}
\end{minipage}
\end{figure}

\begin{definition}[Skeleton$\ast$ property]
Given an interval $J\subset\bS^1$ and a measure $\mu\in\cM_{\rm erg}$, let $h=h(\mu)$ and $\alpha=\chi(\mu)$. We say that $F$ has the \emph{Skeleton property relative to $J$ and $\mu$} if for every finite family of continuous potentials $\varphi_1,\ldots, \varphi_\ell\colon\Sigma_k\times\bS^1\to\bR$  the Skeleton property relative to $J,h$, and $\alpha$ holds true  which in addition satisfies the following property:
\begin{itemize}
\item[(vi)] for every $j=1,\ldots,\ell$
\[
	-K_0 + m(\overline\varphi_j-\varepsilon_E)
	\le \sum_{k=0}^{m-1}\varphi_j\big(F^k(X_i)\big)
	\le K_0 + m(\overline\varphi_j+\varepsilon_E)	,
\]
where $\overline\varphi_j=\int\varphi_j\,d\mu$,
\end{itemize}
with the respective quantifiers.
\end{definition}

In Section~\ref{ss.existenceofskeletons} we will prove that our axioms imply the existence of
skeletons.

\subsection{Skeleton-based hyperbolic sets}

In this section we see that the Axioms CEC$\pm$  together with the Skeleton property provide transitive hyperbolic sets with quite homogeneous properties. The construction of these sets will be based on the so-called multi-variable-time horseshoes built close to skeleton-orbit pieces. These horseshoes will be defined in Section~\ref{sec:mulvarhor}.

\begin{theorem}\label{teo:PES1}
	Consider a transitive skew-product map $F$ as in~\eqref{eq:sp} whose fiber maps are $C^1$. Assume that it satisfies Axiom CEC$+(J)$ for some closed interval $J$ and has the Skeleton property relative to the interval $J$ and some numbers $h>0$
	 and $\alpha\ge0$.
	
	Then for every $\gamma\in(0,h)$ and every $\lambda>0$ there is a compact $F$-invariant topologically transitive hyperbolic set  $\widehat\Gamma\subset\Sigma_k\times\bS^1$ such that
\begin{enumerate}
\item its topological entropy with respect to $F$ satisfies $h_{\rm top}(F,{\widehat\Gamma})\in [h-\gamma,h+\gamma]$ and
\item for every $\nu\in\cM_{\rm erg}({\widehat\Gamma})$ we have $\chi(\nu)\in(\alpha-\lambda,\alpha+\lambda)\cap\bR_+$.
\end{enumerate}
\end{theorem}

We have the following version obtained for the inverse map $F^{-1}$.

\begin{theorem}\label{teo:PES1backward}
	Consider a transitive skew-product map $F$ as in~\eqref{eq:sp} whose fiber maps are $C^1$. Assume that it satisfies Axiom CEC$-(J)$ for some closed interval $J$ and has the Skeleton property relative to $J$ and some numbers $h>0$ and $\alpha\le0$.
	
	Then for every $\gamma\in(0,h)$ and every $\lambda>0$ there is a compact $F$-invariant topologically transitive hyperbolic set  $\widehat\Gamma\subset\Sigma_k\times\bS^1$ such that
\begin{enumerate}
\item its topological entropy with respect to $F$ satisfies $h_{\rm top}({\widehat\Gamma})\in[ h-\gamma,h+\gamma]$ and
\item for every $\nu\in\cM_{\rm erg}({\widehat\Gamma})$ we have $\chi(\nu)\in(\alpha-\lambda,\alpha+\lambda)\cap\bR_-$.
\end{enumerate}
\end{theorem}

Considering a family of continuous potentials instead of the potential $\varphi(\xi,x)=\log\,\lvert f_{\xi_0}'(x)\rvert$, we obtain a more general version of the above result, their proofs are almost identical so we state and will prove them together.

\begin{theorem}\label{teo:PES2}
Consider a transitive skew-product map $F$ as in~\eqref{eq:sp} whose fiber maps are $C^1$. Assume that it satisfies Axiom CEC$+(J)$ for some  closed interval  $J$ and has the Skeleton$\ast$ property relative to $J$ and some measure $\mu\in\cM_{\rm erg}$ with $\chi(\mu)\ge0$ and $h=h(\mu)>0$.
	
	Then for every $\gamma\in(0,h)$, every $\lambda>0$, and every $\kappa>0$ there is a compact $F$-invariant topologically transitive hyperbolic set  $\Gamma$ such that properties 1. and 2. in Theorem~\ref{teo:PES1} are true and in addition $d(\nu,\mu)<\kappa$ for every $\nu\in\cM(\Gamma)$,
	where $d$ is a metric which generates the weak$\ast$ topology.
\end{theorem}

And there is  again an ``inverse version''.

\begin{theorem}\label{teo:PES2backward}
Consider a transitive skew-product map $F$ as in~\eqref{eq:sp} whose fiber maps are $C^1$. Assume that it satisfies Axiom CEC$-(J)$ for some closed interval  $J$ and has the Skeleton$\ast$ property relative to $J$, some measure $\mu\in\cM_{\rm erg}$ with $\chi(\mu)\le0$
and $h=h(\mu)>0$.
	
	Then for every $\gamma\in(0,h)$, every $\lambda>0$, and every $\kappa>0$ there is a compact $F$-invariant topologically transitive hyperbolic set  $\Gamma$ such that properties 1. and 2. in Theorem~\ref{teo:PES1backward} are true and in addition $d(\nu,\mu)<\kappa$ for every $\nu\in\cM(\Gamma)$, where $d$ is a metric which generates the weak$\ast$ topology.
\end{theorem}

We finally state a result that allows us to ``push entropy to the other side'' in the sense that we ``perturb'' an ergodic measure with negative fiber exponent to an ergodic measure with positive exponent. However, comparing with the construction in the proof of the above results,  we 
obtain some lower bound on entropy and some rough estimate of fiber exponent and weak$\ast$ distance (which get worse when considering measures with exponents further away from zero).

\begin{theorem}\label{teo:mimick}
 Consider a transitive skew-product map $F$ as in~\eqref{eq:sp} whose fiber maps are $C^1$. Assume that it satisfies Axioms CEC$+(J)$ and Acc$\pm(J)$ for some closed interval $J\subset\bS^1$. Let $\mu\in \cM_{\rm erg}$ with $\alpha= \chi(\mu) <0$ and assume that $F$ satisfies the Skeleton$\ast$ property relative to $J$ and $\mu$.

Then for every $\beta>0$, $\gamma \in (0,h(\mu))$, $\lambda>0$, and $\varkappa>0$ there is a compact $F$-invariant topologically transitive hyperbolic set $\widehat\Gamma$ such that
\begin{itemize}
\item[1.] its topological entropy satisfies
	\[
		h_{\rm top}(F,\widehat\Gamma)
		\ge \frac{h(\mu) }{1+K_2(|\beta| +\lvert\alpha\rvert) } - \gamma,
	\]
\item[2.] for every $\nu\in \cM_{\rm erg}(\widehat\Gamma)$ we have $\chi(\nu)\beta<0$ and
	\[
		 \frac{ |\beta|}{1+K_2(|\beta|+\lvert\alpha\rvert)}  - \lambda
		 < |\chi(\nu)| <
		 \frac{|\beta|}{1+\frac{1}{\log\D}(|\beta|+\lvert\alpha\rvert)}+ \lambda,
	\]	
	where $K_2$ is as in Axiom CEC$+(J)$ and $\lVert F\rVert$ is as in~\eqref{def:cunifconst}.
\item[3.] for every $\nu\in\cM(\widehat\Gamma)$ we have
	\[
		d(\nu,\mu)
		<\frac{K_2(|\beta|+\lvert \alpha\rvert)}{1+K_2(|\beta|+\lvert \alpha\rvert)}
		+\varkappa.
	\]	
\end{itemize}
The same conclusion is true for $\alpha > 0$ and every $\beta < 0$.
\end{theorem}

The proofs of the above results are postponed to Section~\ref{sec:proofmainprops}. We will only prove Theorems~\ref{teo:PES1} and~\ref{teo:PES2}, the proofs of Theorems~\ref{teo:PES1backward} and~\ref{teo:PES2backward} are analogous, and we will sketch the proof of Theorem~\ref{teo:mimick}.

\subsection{Existence of  skeletons}\label{ss.existenceofskeletons}

In this section we verify skeleton properties under the Axioms CEC$\pm(J)$ and Acc$\pm(J)$ for some interval $J$.

Below we deal
 with the topological entropy of certain sets, that is, to find a separated set of points.
 Since these sets may be noncompact this 
is a bit delicate, recall the definition of entropy in Appendix.

\begin{proposition}
	Consider a skew-product map $F$ as in~\eqref{eq:sp} whose fiber maps are $C^1$ and assume that it satisfies Axioms CEC$\pm(J)$ and Acc$\pm(J)$ for some interval $J\subset\bS^1$.
		Given $\alpha\ge0$, suppose that
	\[
		\cL(\alpha)
		\eqdef
		\big\{(\xi,x)\in\Sigma_k\times\bS^1\colon \lim_{n\to\pm\infty}\frac{1}{n}\log\,
		\lvert (f^n_\xi)'(x)\rvert = \alpha\big\}
		\ne\emptyset,
	\]		
	and let
\[
	h=h_{\rm top}(F,\cL(\alpha)).
\]	

Then the $F$ has the Skeleton property relative to $J,h$, and $\alpha$.	
\end{proposition}

\begin{proof}
Given $J$, consider the constants $m_{\rm f},m_{\rm b}$ provided by Lemma~\ref{lem:idiota}.

Now let $\varepsilon_H\in(0,h)$ and $\varepsilon_E>0$.

We introduce a filtration of the set $\cL(\alpha)$ into sets $\cL_N(\alpha)$ where the finite-time Lyapunov exponents are uniformly $\varepsilon_E$-close to $\alpha$. Given $N\ge1$ define
\[
	\cL_N(\alpha)
	\eqdef \Big\{X=(\xi,x)\in\cL(\alpha)\colon
	\Big\lvert\frac1n\log\,\lvert (f_{[\xi_0\ldots\,\xi_{n-1}]})'(x)\rvert - \alpha\Big\rvert
	\le \varepsilon_E\,\,\forall n\ge N\Big\}.
\]
Note that we have the countable union
\[
	\cL(\alpha)=\bigcup_{N\ge1}\cL_N(\alpha).
\]	
Since entropy is countably stable, for $n_2=n_2(\alpha,\varepsilon_H)$ large enough for every $N\ge n_2$ we have
\[
	h_{\rm top}(F,\cL_N(\alpha))>h-\frac{1}{3}\varepsilon_H.
\]


We have the following intermediate results. Consider the natural projection $\varpi\colon\Sigma_k\times\bS^1\to\Sigma_k$. 

\begin{lemma}\label{lem:eeentropu}
	For any set $\Theta\subset\Sigma_k\times\bS^1$ we have $h_{\rm top}(F,\Theta)=h_{\rm top}(\sigma,\varpi(\Theta))$.
\end{lemma}

\begin{proof}
	Given $\Theta\subset\Sigma_k\times\bS^1$, let $\Theta'=\{\{\xi\}\times\bS^1\colon \xi\in\varpi(\Theta)\}$.
	It is almost straightforward from the definition of entropy to show that $h_{\rm top}(F,\Theta')=h_{\rm top}(\sigma,\varpi(\Theta'))$. This immediately implies the lemma.
\end{proof}

\begin{lemma}\label{lem:uuuposilon}
	Let $\Theta\subset\Sigma_k\times\bS^1$ be a set with entropy $h_{\rm top}(F,\Theta)> h-\varepsilon_H$. Then there is $n_1=n_1(\Theta)\ge1$ such that for every $n\ge n_1$ we have
\[
	M_n(\Theta)\ge e^{n(h-\varepsilon_H)}
	\quad\text{ and }\quad
	M_n(\varpi(\Theta))\ge e^{n(h-\varepsilon_H)},
\]	
where $M_n(\Theta)$ is the maximal cardinality of a set of $(n,1)$-separated points in $\Theta$ and $M_n(\varpi(\Theta))$ denotes the  number for corresponding points in $\varpi(\Theta)$.
\end{lemma}

\begin{proof}
By Lemma~\ref{lem:eeentropu}, $h\eqdef h_{\rm top}(F,\Theta)=h_{\rm top}(\sigma,\varpi(\Theta))$.

Recall the definition of entropy in Appendix. Let us first prove the relation for $M_n(\varpi(\Theta))$. By contradiction, assume that this is not the case and that there is a sequence $n_\ell\to\infty$ such that for every $\ell$ we have
\begin{equation}\label{eq:claimm}
	M_{n_\ell}(\varpi(\Theta))<e^{n_\ell(h-\gamma)}.
\end{equation}
We consider the finite open cover $\cA$ of $\varpi(\Theta)$ by the level-$1$ cylinders $[i]$, $i=0,\ldots,k-1$. For each $\ell$ consider a $(n_\ell,1)$-separated set $\{\xi^{\ell,j}\}_j$ in $\varpi(\Theta)$ of maximal cardinality and their associated  level-$n_\ell$ cylinders $U_{\ell,j}=[\xi^{\ell,j}_0\ldots\xi^{\ell,j}_{n_\ell-1}]$. Observe that, by construction, $\cU_\ell=\{U_{\ell,j}\}_j$ is a cover of $\varpi(\Theta)$. For each $U_{\ell,j}$ we have $n_{\sigma,\cA}(U_{\ell,j})= n_\ell$.  Take any $d>h-\gamma$. Hence, with the estimate~\eqref{eq:claimm} of the cardinality of this cover, we obtain
\[
	\sum_je^{-d\,n_{\sigma,\cA}(U_{\ell,j})}
	\le e^{n_\ell(h-\gamma)} \cdot e^{-d\,n_\ell},
\]
which converges to $0$ when  $\ell\to\infty$. By definition of entropy, this would imply  $h_{\rm top}(\sigma,\varpi(\Theta))\le h-\gamma$, which is a contradiction, proving the first estimate in the lemma.

To prove the second estimate, recall that any pair of $(n,1)$-separated points in $\Sigma_k$ is also $(n,1)$-separated in $\Sigma_k\times\bS^1$.
\end{proof}

Applying the Lemma~\ref{lem:uuuposilon} to $\Theta=\cL_N(\alpha)$, we obtain a number $n_1$. Let $n_0\eqdef \max\{N,n_1\}$, $L_0\eqdef1$, and
\[
	 K_0\eqdef\max_{\ell=0,\ldots,n_0-1}\,\,\max_{X=(\xi,x)\in\Sigma_k\times\bS^1}\,
		\left\{\frac{\lvert (f_{[\xi_0\ldots\,\xi_{\ell-1}]})'(x)\rvert}
			{e^{-\ell(\alpha-\varepsilon_E)}},
			\frac{e^{-\ell(\alpha+\varepsilon_E)}}
			{\lvert (f_{[\xi_0\ldots\,\xi_{\ell-1}]})'(x)\rvert}\right\}.
\]

Hence, for every $m\ge n_0$, again by the lemma, we obtain a set $\fX=\{X_i\}\subset\cL_N(\alpha)$ of $(m,1)$-separated points $X_i=(\xi^i,x_i)$ which has cardinality $\card\fX\ge e^{m(h-\varepsilon_H)}$.  Clearly, the choice of this set $\fX$ depended on $h,\alpha,\varepsilon_H,\varepsilon_E$, and $m$, only.	Thus, we already checked  (i)  in the Skeleton property.

By construction, the sequences $\xi^i$ are $(m,1)$-separated and, in particular, this implies (ii).

Further, the choice of $K_0$ above implies (iii).
Finally, (iv) and (v) follow from Lemma~\ref{lem:idiota}.
This completes the proof of the proposition.
\end{proof}

\begin{proposition}\label{prop:lunch1}
	Consider a skew-product map $F$ as in~\eqref{eq:sp} whose fiber maps are $C^1$ and assume that it satisfies Axioms CEC$\pm(J)$ and Acc$\pm(J)$ for some interval $J\subset\bS^1$.	Let $\mu\in\cM_{\rm erg}$.

Then the Skeleton$\ast$ property relative to $J$ and $\mu$ holds.
\end{proposition}

\begin{proof}
Let $\varphi_1,\ldots,\varphi_\ell\colon\Sigma_k\times\bS^1\to\bR$ be continuous functions.  Put
\[
	\alpha=\chi(\mu),
	\quad h=h(\mu).
\]

Given $J$, consider the constants $m_{\rm f},m_{\rm b}$ provided by Lemma~\ref{lem:idiota}.
By  this lemma,
 there is some $r\le m_{\rm f}$ and a finite sequence $(\theta_1\ldots\theta_{r})$ such that
\[
	\mu\big(\Sigma_k\times
	f_{[\theta_1\ldots\,\theta_{r}]}(J)\big)>0.
\]
Let $I\eqdef f_{[\theta_1\ldots\theta_{r}]}(J)$.
By this lemma there is also some $s\le m_{\rm b}$ and a finite sequence $(\beta_1\ldots\beta_s)$ such that
\[
	\mu\big(\Sigma_k\times f_{[\beta_1\ldots\,\beta_s.]}(J)\big)>0.
\]
Let $I'\eqdef f_{[\beta_1\ldots\,\beta_s.]}(J)$. Hence, as both sets are of positive measure, by ergodicity
of $\mu$, there is some finite sequence $(\delta_0\ldots\delta_{\ell-1})$ such that
\[
	\mu\big(\Sigma_k\times \big( f_{[\delta_0\ldots\,\delta_{\ell-1}]}(I)\cap I'\big)\big)>0.
\]
Let now $I ''\eqdef  f_{[\delta_0\ldots\,\delta_{\ell-1}]}(I)\cap I'$ and  $A\eqdef\Sigma_k\times I''$. To simplify notation, we continue to denote $(\theta_1\ldots\,\theta_r\delta_0\ldots\delta_{\ell-1})$  by $(\theta_1\ldots\,\theta_r)$.

Fix $\kappa\in(0,\mu(A)/4)$ and $t\in(0,1)$.

By Proposition~\ref{pro:BriKat} applied to $A=\Sigma_2\times I''$ there are $n_0=n_0(\kappa,\gamma)\ge1$, $K_0=K_0(\kappa)>1$, and a set $\Lambda'\subset\Lambda$ satisfying $\mu(\Lambda')>1-\kappa$ such that for every $n\ge n_0$ there is $m\in\{n,\ldots,n(1+t)\}$ and a set  of $(m,1)$-separated points $\{X_i\}\subset A\cap\Lambda'$, $X_i=(\xi^{i},x_i)$, satisfying
\[
	x_i\in I'',\quad
	f_{[\xi_0^{i}\ldots\,\xi_{m-1}^{i}]}(x_i)\in  I''.
\]
Moreover, this set has cardinality
\[
	\card\{x_i\}
	\ge \big(\mu(A)-\kappa\big) \cdot e^{m(h(\mu)-\gamma)}
\]
Letting $L_0\eqdef 1/(\mu(A)-\kappa)$, this shows items (i) and (ii) of the Skeleton$\ast$ property.

Moreover,
for every $n\ge0$ we have
\[
	K_0^{-1}e^{n(\alpha-\varepsilon_E)}
	\le 
	\lvert (f_{[\xi_0^i\ldots\,\xi_{n-1}^i]})'(x_i)\rvert
	\le K_0e^{n(\alpha+\varepsilon_E)},
\]
which shows item (iii).
And for every $j=1,\ldots,\ell$ and $n\ge0$ we have
\[
	-K_0+n(\overline\varphi_j-\varepsilon_E)
	\le \sum_{k=0}^{n-1}\varphi_j(f_{[\xi_0^i\ldots\,\xi_{n-1}^i]}(x_i))
	\le K_0+n(\overline\varphi_j+\varepsilon_E),
\]
which shows item (vi).

By the choice of $I''$, we have
\[
	x_i'
	= f_{[\theta_1\ldots\,\theta_r]}^{-1}(x_i)\in J,
\]
which implies item (iv). By our choice of  $(\beta_1\ldots\beta_s)$, $k\le m_{\rm b}$, such that
\[
	f_{[\beta_1\ldots\,\beta_s]}\big(f_{[\xi_0^i\ldots\,\xi_{k-1}^i]}(x_i)\big)
	= f_{[\xi_0^i\ldots\,\xi_{k-1}^i\beta_1\ldots\,\beta_s]}(x_i)\in J,
\]
which implies item (v). This finishes the proof of the  Skeleton$\ast$ property.
\end{proof}

\section{Multi-variable-time horseshoes}\label{sec:mulvarhor}

In this section we introduce multi-variable-time horseshoes which will provide the essential pieces of our construction.
Here we refer to a concept similar to the ``interval" horseshoes in the sense of Misiurewicz and Szlenk~\cite{MisSzl:80} rather than the ``standard'' one in the sense of Smale. The connection with the skeletons from Section~\ref{sec:skeletons} will become clear in  Section~\ref{sec:proofmainprops}.
A key step is to estimate the entropy of these objects, see Proposition~\ref{prop:wilvarhorse}.

\begin{definition}[Multi-variable-time horseshoes]
Let $X$ be a compact metric space and $T\colon X\to X$ a local homeomorphism. \\[0.1cm]
\textbf{Markov rectangles and transition maps.} Let $\{S_i\}_{i=1}^M$ be a family of disjoint compact subsets of $X$ that we call \emph{Markov rectangles}. Assume that there are positive integers $t_{\rm min},t_{\rm max}$, $t_{\rm min}\le t_{\rm max}$, such that for every  $i,j\in\{1,\ldots,M\}$
\begin{itemize}
\item there exists a \emph{transition times} $t_{ij} \in \{t_{\rm min},\ldots,t_{\rm max}\}$ so that $S_j \subset T^{t_{ij}}(S_i)$ and
\item  the \emph{transition map} $T^{t_{ij}}|_{S_i\cap T^{-t_{ij}}(S_j)}$ is injective.
\end{itemize}
Let (compare Figure~\ref{fi.neuuhorse})
\[
	S_{ij}
	\eqdef T^{-t_{ij}}(S_j).
\]
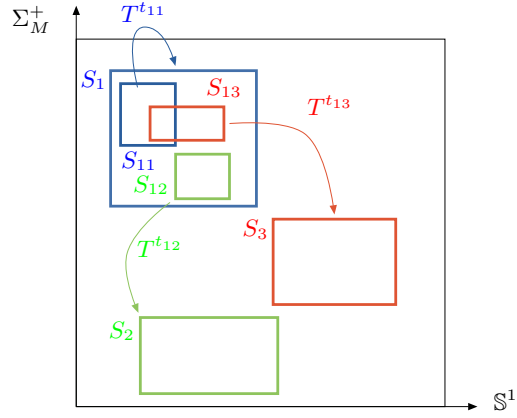
\begin{figure}
\begin{minipage}[h]{\linewidth}
\centering
 \begin{overpic}[scale=.35]{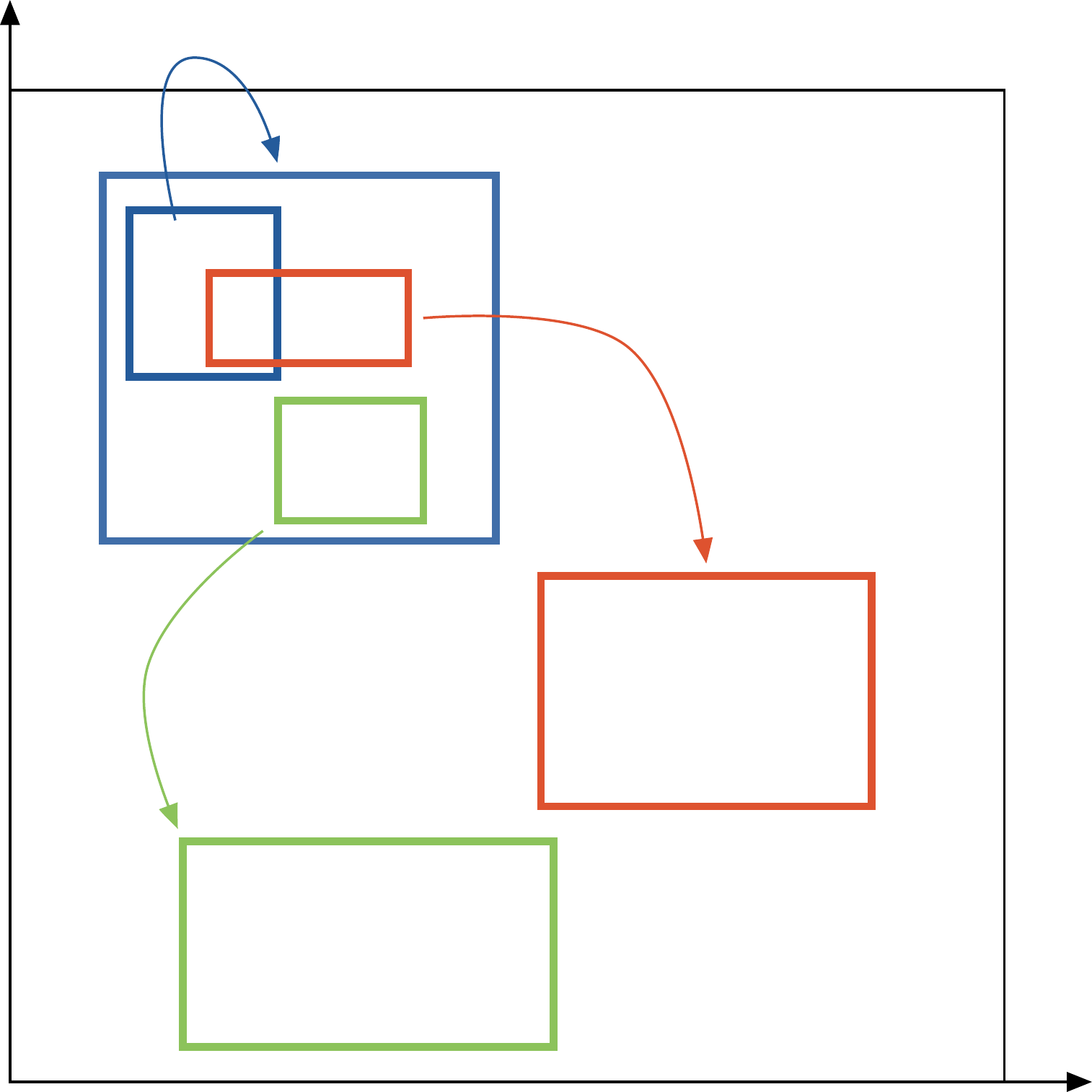}
  			\put(104,0){\small$\bS^1$}
  			\put(-15,95){\small$\Sigma_M^+$}
  			\put(2,80){\small\textcolor{blue}{$S_1$}}
  			\put(9,18){\small\textcolor{green}{$S_2$}}
  			\put(42,43){\small\textcolor{red}{$S_3$}}
  			\put(33,78){\small\textcolor{red}{$S_{13}$}}
  			\put(58,73){\small\textcolor{red}{$T^{t_{13}}$}}
  			\put(12,96){\small\textcolor{blue}{$T^{t_{11}}$}}
  			\put(12,60){\small\textcolor{blue}{$S_{11}$}}
  			\put(15,54){\small\textcolor{green}{$S_{12}$}}
  			\put(16,38){\small\textcolor{green}{$T^{t_{12}}$}}
 \end{overpic}
\caption{Construction of the multi-variable-time horseshoe, before choosing admissible transitions.}
\label{fi.neuuhorse}
\end{minipage}
\end{figure}	

\noindent\textbf{Coding the allowed transitions.}
Consider a subshift of finite type $\Sigma_A\subset\{1,\ldots,M\}^\bZ$ defined as follows:
For each $i$ let $t=t(i)$ be a length%
\footnote{This choice is in general not unique but this fact is inessential for our purposes, see Remark~\ref{rem:remakrr}} %
 for which the number of $j$'s  with transition times of length $t_{ij}=t$ is maximal. Consider the
 transition matrix $A=(a_{ij})_{i,j=1}^M$, where
\begin{equation}\label{eq:transitions}
	a_{ij}\eqdef
	\begin{cases}
		1&\text{ if } t_{ij}=t(i),\\
			0&\text{ otherwise}.
	\end{cases}
\end{equation}
This defines a transition matrix for a subshift of finite type, that is, the subset of all
\emph{$A$-admissible} sequences:
\[
	\Sigma_A
	\eqdef \big\{(c_1c_2\ldots)\in\{1,\ldots,M\}^\bZ\colon
		a_{c_i c_{i+1}}=1\text{ for all }i\ge1\big\}.
\]	
\textbf{Construction of the invariant set.}
For a 
given finite $A$-admissible sequence 
 $(c_0\ldots c_{n-1})$ we
define inductively
\[
	S_{c_0\ldots c_{k-1}}
	\eqdef T^{-t_{c_{k-2}c_{k-1}}}(S_{c_0\ldots c_{k-2}}).
\]
Let
\[
\Gamma'
	\eqdef \bigcap_{n\ge1}\,\bigcup_{[c_0\ldots\,c_{n-1}]}
	S_{c_0\ldots c_{n-1}},
\]
where the union is taken over all $n$th level cylinders of $\Sigma_A$.
Let
\[
	\Gamma
	\eqdef \bigcup_{k=0}^{t_{\rm max}-1} T^k(\Gamma').
\]
We call  $(T,\Gamma)$ a \emph{multi-variable-time horseshoe}.
\end{definition}

Below we will prove the following proposition:

\begin{proposition}\label{prop:wilvarhorse}
Let $X$ be a compact metric space and $T\colon X\to X$ a local homeomorphism.
Assume that $(T,\Gamma)$ is a multi-variable-time horseshoe as above.
	Then the topological entropy of $T\colon\Gamma\to\Gamma$ satisfies
\[
	 \frac{\log M - \log(t_{\rm max}-t_{\rm min}+1)}{ t_{\rm max}}\le 
	h_{\rm top}(T,\Gamma)
	\le \frac{\log M}{t_{\rm min}}.
\]	
\end{proposition}

Before going to the details of the proof of this proposition we make some comments
on the above definition.

\begin{remark}\label{rem:remakrr}
Note that this subshift $\Sigma_A$ is not necessarily transitive. However, there exists%
\footnote{Recall that the subshift of finite type is represented by a transition graph and that transitive invariant subsets correspond to  irreducible components in this graph  (see, for example,~\cite[Chapter 4.4]{LinMar:95}). Since the graph has finitely many edges only, there are at most finitely many such components. Now apply countable stability of entropy, see~\eqref{eq:entcounstab}.}
 a subshift of finite type of $\Sigma_A$ which is transitive and has the same topological entropy, hence we will for simplicity assume that $\Sigma_A$ is transitive.

 We remark that our choice of $A$ is not unique: given $i$ there could exist more than one length $t$ for which the number of $j$'s with transition time $t$ is maximal. But any choice would not alter our estimates of entropy.
\end{remark}

\begin{remark}
Let us explain the reason to study some subsystem only. Clearly, the inverse maps $T^{-t_{ij}}\colon S_j\to S_{ij}=S_i\cap T^{-t_{ij}}(S_j)$ are all well-defined. However, the sets $S_{ij}$ are, in general, not pairwise disjoint%
\footnote{Consider, for example, the two one-point sets $S_1=\{P\}$, $S_2=\{Q\}$ and the map $T\colon\{P,Q\}\to\{P,Q\}$ defined by $T(P)=Q$, $T(Q)=P$ and let $t_{11}=t_{22}=2$, $t_{12}=t_{21}=1$.
	The map $T$ has entropy zero.	
	The associated multi-variable-time horseshoe also does.
	Note that the bound in Proposition~\ref{prop:wilvarhorse} is sharp in this case.}%
. Or choice of admissible sequences~\eqref{eq:transitions} guarantees that for every $A$-admissible pairs $ij$ and $i\ell$, $j\ne\ell$, the sets $S_{ij}$ and $S_{i\ell}$ are indeed disjoint  and that hence our symbolic description of the horseshoe is indeed well-defined. For example, in Figure~\ref{fi.neuuhorse}, $t_{11}\ne t_{13}$ and at most one of the transitions $11$ or $13$ will be admissible in the sense of~\eqref{eq:transitions}.
\end{remark}

\begin{remark}
	To relate the above defined object to other contexts, note that the various transition times $t_{ij}$ for $A$-admissible transitions $ij$ can be related to a so-called \emph{jump transformation} $\fT\colon\Gamma'\to\Gamma'$ by setting
\[
	\fT(x)
	\eqdef T^{t(i)}(x)
	\quad\text{ for every }\quad
	x\in\Gamma'\cap S_i,
\]
generalizing the classical concept of \emph{first return maps}. Such transformations were considered, for example, by Schweiger~\cite{Sch:81}.
By construction, $(\sigma,\Sigma_A^+)$ is conjugate to $(\fT,\Gamma')$ via $\pi_A$.

The multi-variable-time horseshoe is a more general version of the variable-time horseshoe of Luzzatto and Sanchez-Salas~\cite{LuzSan:13}. In their approach, given $i$, the transition times are $t_{ij}=t(i)$  constant for all $j$.
\end{remark}

\begin{remark}
It is useful to introduce a symbolic description of the system.
By the above, for every $x\in\Gamma'$ there is a unique sequence $\underline c\in\Sigma_A^+$ such that $x\in \bigcap_{n\ge1}	S_{c_0\ldots c_{n-1}}$.
Hence, we have naturally given a projection
\[
	\pi_A\colon \Gamma'\to\Sigma_A^+,\quad
	\pi(x)\eqdef \underline c
	\quad\text{ if } \quad
	x\in  \bigcap_{n\ge1}	S_{c_0\ldots c_{n-1}}.
\]

Observe that we do not assume that $T$ is expansive. In particular, there might exist a symbolic sequence $\underline c$ which corresponds to more than one point in $\Gamma$.

Naturally, the symbolic description applies to the set $\Gamma'$ only, which is invariant under the jump transformation only.

Finally, note that $\Gamma\subset X$ is a compact $T$-invariant set. Indeed, by construction, $\Gamma$ is the forward orbit of the set $\Gamma'$ by $T$.
It is the smallest $T$-invariant set containing $\Gamma'$.
\end{remark}

\begin{proof}[Proof of Proposition~\ref{prop:wilvarhorse}]
We begin with the following auxiliary result. Let $d$ be the minimum of the Hausdorff distance between the rectangles $S_i$.

\begin{lemma}
	Given a pair of $A$-admissible sequences $\underline c,\underline c'$, for  every $\ell\ge0$ with $c_n\ne c_n'$ for some $n\in\{0,\ldots,\ell-1\}$, every pair of points $x\in S_{c_0\ldots c_{\ell-1}}$ and
	$y \in S_{c_0'\ldots c_{\ell-1}'}$ are $(n\, t_{\rm{max}},d)$-separated.
\end{lemma}

\begin{proof}
If $c_0\ne c_0'$ then the orbits are (at least) $d$-separated (they start in points which are in disjoint sets $S_i$ and thus at distance at least $d$).
Otherwise, if $c_0=c_0'$ and $c_1\ne c_1'$ then orbits get separated after time $t_{c_0c_1}$ (observe that by our choice of admissible transitions $A$ we have $t_{c_0c_1}=t_{c_0'c_1'}$) and hence are $(t_{\rm max},d)$-separated. To finish the proof,  continue by induction on $\ell$.
\end{proof}

Continuing with the proof of the proposition,
observe that, by the pigeonhole principle, for every $i\in\{1,\ldots,M\}$ the number of indices $j$ with $a_{ij}=1$ is bounded from below by $M/(t_{\rm max}-t_{\rm min}+1)$ and, trivially, from above by $M$. This immediately implies that the entropy satisfies
\begin{equation}\label{eq:laaabel}
	\log \frac{M}{t_{\rm max}-t_{\rm min}+1}\le 
	h_{\rm top}(\sigma,\Sigma_A^+)
	\le \log M.
\end{equation}		
Recall that we always have
\[
	h_{\rm top}(T,\Gamma)
	\ge \liminf_{n\to\infty}\frac{1}{n}\log s_n(d,\Gamma'),
\]
where $s_n(d,\Gamma')$ denotes the maximal cardinality of a $(n,d)$-separated set of points in $\Gamma'$. Denote by $s_n(1,\Sigma_A^+)$ the analogous number, that is, the number of distinct cylinders of level $n$ in $\Sigma_A^+$.
Observe that $s_{nt_{\rm max}}(d,\Gamma')\ge s_n(1,\Sigma_A^+)$. Therefore,
\[	
	h_{\rm top}(T,\Gamma)
	\ge \liminf_{m\to\infty}\frac{1}{m\,t_{\rm max}}\log s_m(1,\Sigma_A^+)
	= \frac{1}{t_{\rm max}}h(\sigma|_{\Sigma_A^+}),
\]	
To obtain the equality, just recall that $\sigma|_{\Sigma_A^+}$ is a subshift of finite type. The lower bound  now follows from~\eqref{eq:laaabel}

To get the upper bound, note that there is a universal constant $K$ such that for every $(n,d)$-separated set of points in $\Gamma$ there is a $(n+t_{\rm max},Kd)$-separated set in $\Gamma'$. With this in mind, similarly we can conclude
\[
	h_{\rm top}(T,\Gamma)
	\le \frac{1}{t_{\rm min}}\log M.	
\]
This proves the proposition.
\end{proof}

\section{Multi-variable-time horseshoes in our setting}\label{sec:proofmainprops}

In this section, we consider a step skew-product $F$ with $k$ circle fiber maps 
as in~\eqref{eq:sp} and prove Theorems~\ref{teo:PES1}--\ref{teo:mimick}.

\subsection{Main ingredients and sketch of the construction}\label{sec:sketch}
The general idea is to use the Skeleton property together with CEC$+$ to construct a multi-variable-time horseshoe with the desired properties. 
In the entire Section~\ref{sec:proofmainprops} we consider an interval $J\subset\bS^1$ such that Axiom CEC$+(J)$ holds with associated constants $K_1,\ldots,K_5$. Moreover, we assume that the Skeleton property holds relative to $J$ and numbers $h\ge0$ and $\alpha\ge0$ with connecting times $m_{\rm b},m_{\rm f}\ge1$. For any set of appropriate quantifiers, this gives a skeleton $\fX=\{X_i\}$, $X_i=(\xi^i,x_i)$.

We first briefly sketch the construction of  multi-variable-time horseshoes which will provide the claimed transitive hyperbolic sets, more details are given below.

Items (iv) and (v) of the Skeleton property provide us with a  family of  itineraries
\begin{equation}\label{eq:itinin}
	(\zeta^i_0\ldots\zeta^i_{m_i-1})
	\eqdef
	(\theta^i_1\ldots\theta^i_{r_i}\xi^i_0\ldots\xi^i_{m-1}\beta^i_1\ldots\beta^i_{s_i})
	,\quad
	m_i
	\eqdef r_i+m+s_i,
\end{equation}
where $r_i\le m_{\rm f},s_i\le m_{\rm b}$ and a family of points $\{x_i\}\subset\bS^1$ and $\{x_i'\}\subset J$ satisfying
\begin{equation}\label{eq:saccomerda}
	f_{[\theta^i_1\ldots\,\theta^i_{r_i}]}(x_i')=x_i
	\quad\text{ and }\quad
	f_{[\xi^i_0\ldots\,\xi^i_{m-1}\beta^i_1\ldots\,\beta^i_{s_i}]}(x_i)\in J,
\end{equation}
where $m$ will be sufficiently big and specified below. For certain $\varepsilon_H$, we also have
\[
		M\eqdef\card\fX
	\ge L_0^{-1}e^{m(h-\varepsilon_H)}.
\]
Without loss of generality, as we can simply disregard some sequences, we can assume 
\begin{equation}\label{eq:saccoentropy}
	L_0^{-1}e^{m(h-\varepsilon_H)}\le M\eqdef\card\fX\le L_0e^{m(h+\varepsilon_H)} .
\end{equation}

To obtain the multi-variable-time horseshoe
we consider the compact metric space $\Sigma_k^+\times\bS^1$
and
 for each
 $i\in\{1,\ldots,M\}$ we consider the Markov rectangles
\begin{equation}\label{eq:MarRec}
	S_i
	= [\xi^i_0\ldots\xi^i_{m-1}]\times I_i,
\end{equation}
where $I_i\subset\bS^1$ are sufficiently small (according to controlled distortion) intervals each centered at its corresponding point $x_i$.
Axiom CEC$+(J)$  gives finite expanding and covering sequences $(\eta^i_0\ldots\eta^i_{\ell_i-1})$. 
We consider the projection to  ``the unstable direction''
\[
	\widehat\pi\colon\Sigma_k\times\bS^1\to\Sigma_k^+\times\bS^1,
	\quad\widehat\pi(\xi^-.\xi^+,x)\eqdef (\xi^+,x)
\]
and let
\[
	T\colon X\to X,\quad T\eqdef\widehat\pi\circ F.
\]	
Then we consider the  concatenations  and transition times
\begin{equation}\label{e.carnaval}
	(\xi^i_0\ldots\xi^i_{m-1}\beta^i_1\ldots\beta^i_{s_i}
		\eta^i_0\ldots\eta^i_{\ell_i-1}\theta^j_1\ldots\theta^j_{r_j}),\quad
	t_{ij}\eqdef m+s_i+\ell_i+r_j
\end{equation}
having the covering property
\begin{equation}\label{eq:cobrindo}
	T^{t_{ij}}(S_i)
	= \Sigma_k^+\times
		f_{[\xi^i_0\ldots\,\xi^i_{m-1}\beta^i_1\ldots\,\beta^i_{s_i}
			\eta^i_0\ldots\,\eta^i_{\ell_i-1}\theta^j_1\ldots\,\theta^j_{r_j}]}(I_i)
	\supset \Sigma_k^+\times I_j
	\supset S_j.
\end{equation}
This defines a multi-variable-time horseshoe $\Gamma\subset\Sigma_k^+\times\bS^1$.
 This set symbolically extends in a unique way to a compact $F$-invariant set $\widehat\Gamma\subset\Sigma_k\times\bS^1$ which will give the set in the theorems.
\smallskip

We now sketch the content of the following subsections which contain the steps to prove Theorems~\ref{teo:PES1}--\ref{teo:PES2}.
We first explain a bit more precisely the construction of the multi-variable-time horseshoes using the Skeleton property. Then we will estimate from below their entropy based on Proposition~\ref{prop:wilvarhorse}.
The estimate of the exponent (and Birkhoff averages of a family of potentials) will also be a result of the explicit construction.

In Section~\ref{sec:quantify} we select some quantifiers.
In Section~\ref{ssec:selecting} we choose the finite sequences $(\theta^i_1\ldots\theta^i_{r_i}),(\xi^i_0\ldots\xi^i_{m-1}),(\beta^i_1\ldots\beta^i_{s_i}),$ and $(\eta^i_1\ldots\eta^i_{\ell_i})$ in~\eqref{e.carnaval}. We also choose  some intermediate intervals $I_i '$ which eventually lead to the definition of the intervals $I_i$.  Concatenating appropriate blocks of finite sequences as in~\eqref{e.carnaval}, we define the  maps $T^{t_{ij}}$ (see Figure~\ref{fig.iterates}). This will complete the definition of the Markov rectangles and transition maps, that is, of the multi-variable-time horseshoes.
%
In Section~\ref{ssec:exponenting} we estimate the Lyapunov exponents.
In Section~\ref{ssec:entropying} we estimate the topological entropy from below. In Section~\ref{ssec:birkhoff} we estimate Birkhoff averages in order to derive the weak$\ast$ approximation. At the end of Section~\ref{sec:proofmainprops} we complete the proofs of Theorems~\ref{teo:PES1} and~\ref{teo:PES2}.


\subsection{Choosing quantifiers}\label{sec:quantify}

As a reference point, the quantifiers are chosen in the following order: Given the interval $J$  and numbers $h,\alpha,\gamma,\lambda$, we choose $\varepsilon,$ $\varepsilon_E,\varepsilon_H$, and $\varepsilon_D$. Then we will choose $m$ sufficiently large which allows us to choose $\delta_0$ small.

Given $h>0$ and $\alpha\ge0$, fix $\gamma\in(0,h)$ and $\lambda>0$.

We let $\varepsilon_H,\varepsilon_E$ much smaller than $\gamma,\lambda$. Associated to these numbers, by the Skeleton property, there are $K_0,L_0\ge1$ and $n_0\ge1$ such that for every $m\ge n_0$ there is a skeleton $\fX=\fX(h,\alpha,\varepsilon_H,\varepsilon_E,m)=\{X_i\}$ with $X_i=(\xi^i,x_i)$. Note that $m$ can be chosen arbitrarily large.

Recall the definition of the universal constant $\D$ of the IFS  in~\eqref{def:cunifconst}
and the numbers $m_{\rm b}$ and $m_{\rm f}$ in the Skeleton property. 
Let
\begin{equation}\label{eq:K0hat}
	\widehat K_0
	\eqdef K_0 \D^{m_{\rm b}+m_{\rm f}}.
\end{equation}

We now take $\varepsilon_D>0$ and $\delta_0>0$ satisfying the assumption in 
distortion Corollary~\ref{cor:distortionC1} applied to $\varepsilon_D$ such that
\begin{equation}\label{defcond:K1K4delta0}
\delta_0 < \exp(-m \sqrt{\varepsilon_D + \varepsilon_E\,}).
\end{equation}
In particular, for sufficiently large $m$, the number $\delta_0$ is much smaller than
$\gamma$ and $\lambda$ and also satisfies
\begin{equation}\label{defcond:K1K4}
	\delta_0
	<\min\{K_1, K_4\}.
\end{equation}
Moreover, for sufficiently large $m$, there is some $\vartheta>0$ such that we have
\begin{equation}\label{defcond:Kfinal}
	K_2K_5\lvert\log\delta_0\rvert - m(\varepsilon_E+\varepsilon_D)
		- \log \widehat K_0
		+ K_3K_5	
	\ge\vartheta	
\end{equation}
From~\eqref{defcond:K1K4delta0}, when $\varepsilon_D$ was initially chosen sufficiently small  and $m$ is sufficiently large, we can also guarantee that
\begin{equation}\label{defcond:vareps1}
	\varepsilon_1
	\eqdef \frac{\lvert\log \delta_0\rvert}{m}
	\ll \min\{\gamma,\lambda\}.
\end{equation}

\subsection{Choosing the Markov rectangles and transition times of the horseshoe}\label{ssec:selecting}

We now specify the rectangles~\eqref{eq:MarRec}. Recall the family of points $\{x_i\}\subset\bS^1$ and $\{x_i'\}\subset J$ given in~\eqref{eq:saccomerda} and itineraries in~\eqref{eq:itinin}
obtained from the Skeleton property.
\begin{figure}
\begin{minipage}[h]{\linewidth}
\centering
 \begin{overpic}[scale=.50]{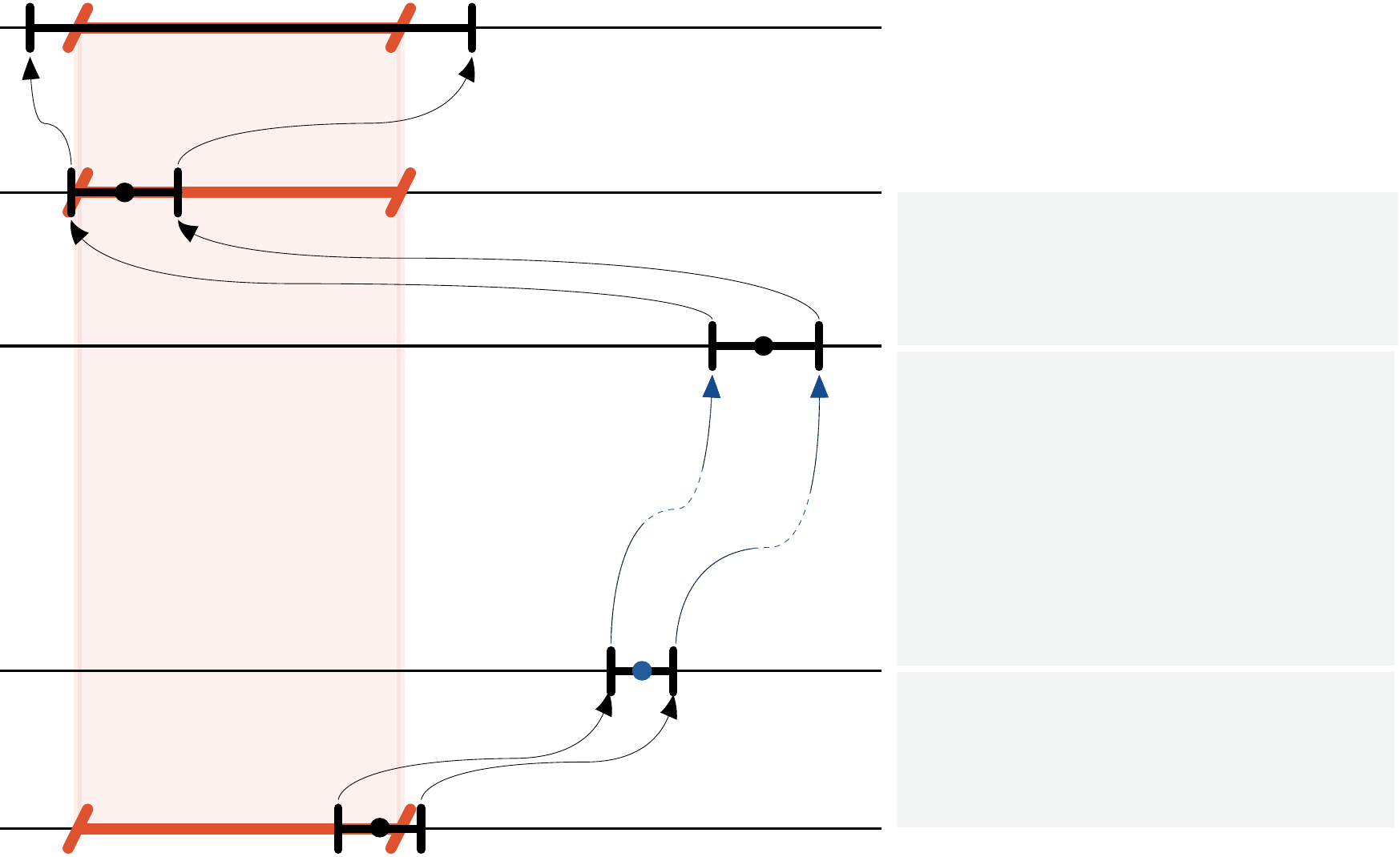}
 			\put(-20,8){\small$f_{[\theta^i_1\ldots\,\theta^i_{r_i}]}$}
 			\put(-20,24){\small$f_{[\xi^i_0\ldots\,\xi^i_{m-1}]}$}
 			\put(-20,42){\small$f_{[\beta^i_1\ldots\,\beta^i_{s_i}]}$}
 			\put(-20,54){\small$f_{[\eta^i_0\ldots\,\eta^i_{\ell_i-1}]}$}
  			\put(26,-3){\small$x_i'$}
  			\put(31,-2){\small$I_i'$}
  			\put(44.5,9.5){\small$x_i$}
  			\put(49.5,9.5){\small$I_i$}
  			\put(0,43){\small$H_i'$}
  			\put(66,25){\small$(m,1)$-separated}
			\put(66,21){\small orbit pieces}
  			\put(66,53){\small CEC$+(J)$}
  			\put(66,42){\small Acc$-(J)$}
  			\put(66,6){\small Acc$+(J)$}
  			\put(-4,1){\small$\bS^1$}
  			\put(6,-3){\small\textcolor{red}{$J$}}
			\put(101,3){\small Skeleton}
  \end{overpic}
\caption{}
\label{fig.iterates}
\end{minipage}
\end{figure}
Define the auxiliary intervals
\begin{equation}\label{defIiprime}
	I_i'
	\eqdef B\big(x_i',\delta_0\widehat K_0^{-1}e^{-m(\alpha+\varepsilon_E+\varepsilon_D)}\big)
\end{equation}
which finally lead to  the intervals $I_i$
defining   the rectangles~\eqref{eq:MarRec}, compare Figure~\ref{fig.iterates}.
We now are going to verify that with this choice, these rectangles have the covering property~\eqref{eq:cobrindo}. By doing so, we will also collect some estimates needed to estimate entropy and exponents.

Recall the definition of $\widehat K_0$ in~\eqref{eq:K0hat}
and the skeleton sequences $(\zeta^i_0\ldots\,\zeta^i_{m_i-1})$ 
in \eqref{eq:itinin}, with $m_i=m+s_i+r_i$ and  $r_i\le m_{\rm f},s_i\le m_{\rm b}$.
By Corollary~\ref{cor:distortionC1}, for every $x\in I_i'$ we have
\begin{equation}\label{eq:minmax}
	 \widehat K_0^{-1}e^{m(\alpha-\varepsilon_E-\varepsilon_D)}
	\le \lvert(f_{[\zeta^i_0\ldots\,\zeta^i_{m_i-1}]})'(x)\rvert
	\le  \widehat K_0e^{m(\alpha+\varepsilon_E+\varepsilon_D)}.
\end{equation}
Let now
\[
	H_i'
	\eqdef f_{[\zeta^i_0\ldots\,\zeta^i_{m_i-1}]}(I_i')
\]
and observe that this interval intersects $J$, recall~\eqref{eq:saccomerda}.
We also observe that
\[
	\lvert H_i'\rvert
	\sim \lvert I_i'\rvert \cdot
		\lvert (f_{[\zeta^i_0\ldots\,\zeta^i_{m_i-1}]})'(x_i)\rvert
\]
up to a multiplicative factor due to distortion, indeed it follows from~
\eqref{defIiprime} and
\eqref{eq:minmax}  that
\begin{equation}\label{eq:something}
	\delta_0\widehat K_0^{-2}e^{-2m(\varepsilon_E+\varepsilon_D)}
	\le \lvert H_i'\rvert
	\le \delta_0
	<K_1,
\end{equation}
where we also used~\eqref{defcond:K1K4}.
Thus, we can apply Axiom CEC$+(J)$ to each interval $H_i'$.

Observe that by Lemma~\ref{lem:warsaw} there is  a subinterval $\widehat H_i'\subset H_i'$ having the covering and expansion properties for an iteration length which is in fact bounded from above and below. Therefore, without loss of generality and for notational simplicity, we can assume  $\widehat H_i'= H_i'$ and that there are sequences $(\eta^i_0\ldots\eta^i_{\ell_i-1})$
such that
\begin{equation}\label{eq:coveringcov}
	f_{[\eta^i_0 \ldots\eta^i_{\ell_i-1}]}(H_i')
	\supset B(J,K_4)
	\supset \bigcup_{i=1}^{M=\card\fX}I_i',
\end{equation}
where the last inclusion follows from~\eqref{defcond:K1K4}, with integers $\ell_i$ satisfying
\begin{equation}\label{eq:lorenzito}
	K_2\,\lvert\log\lvert H_i'\rvert\rvert +K_3
	\le \ell_i
	\le 2(K_2\,\lvert\log\lvert H_i'\rvert\rvert +K_3 ).
\end{equation}

Finally, based on the sequences of the skeleton, see \eqref{eq:itinin}, 
we define the rectangles by
\[
	I_i
	\eqdef f_{[\theta^i_1\ldots\,\theta^i_{r_i}.]}(I_i')
	\quad\text{ and }\quad
	S_i
	= [\xi^i_0\ldots\xi^i_{m-1}]\times I_i.
\]
By our choices, for every pair $i,j\in\{1,\ldots,M\}$, with
\[
	 (\xi^i_0\ldots\xi^i_{m-1}\beta^i_1\ldots\beta^i_{s_i}\eta^i_0\ldots\eta^i_{\ell_i-1}\theta^j_1\ldots\theta^j_{r_j}),\quad
	t_{ij}\eqdef m+s_i+\ell_i+r_j
\]
by the covering property~\eqref{eq:coveringcov}, we have
\[
	F^{t_{ij}}(S_i)
	\supset \Sigma_k^+\times
		f_{[\xi^i_0\ldots\,\xi^i_{m-1}\beta^i_1\ldots\,\beta^i_{s_i}
			\eta^i_0\ldots\,\eta^i_{\ell_i-1}\theta^j_1\ldots\,\theta^j_{r_j}]}(I_i)
	\supset \Sigma_k^+\times I_j
	\supset S_j.
\]
Thus, we obtain the desired covering property in the hypotheses of the 
definition of a multi-variable-time horseshoe.

\subsection{Controlling Lyapunov exponents}\label{ssec:exponenting}

For the sequel, we need some preliminary estimates. We observe that from~\eqref{eq:lorenzito} and~\eqref{eq:something} we have
\begin{equation}\label{eq:ellell2}
	K_2\lvert\log\delta_0\rvert+K_3
	\le \ell_i	
	\le 2\Big(K_2\big(\lvert\log(\delta_0\widehat K_0^{-2})\rvert+2m(\varepsilon_E+\varepsilon_D)\big)+K_3\Big).
\end{equation}
Hence,
\[\begin{split}
	\ell_i	
	&\le 2\Big(K_2\big(\lvert\log(\delta_0\widehat K_0^{-2})\rvert+2m(\varepsilon_E+\varepsilon_D)\big)+K_3\Big)\\
\text{with~\eqref{defcond:vareps1}}\quad
	&=  2\Big(K_2\big(2\log\widehat K_0+m(\varepsilon_1+2\varepsilon_E+2\varepsilon_D)\big)+K_3\Big)
\end{split}\]
and
\[
	 K_2 m\varepsilon_1 +K_3  \le \ell_i .
\]
From this, using the Landau symbol and recalling again that 
 $m_i=m+s_i+r_i$ where  $r_i\le m_{\rm f},s_i\le m_{\rm b}$,
 we can conclude
\begin{equation}\label{eq:formulaa}
	\frac{m}{m_i+\ell_i}
	= 1 + O\Big(\varepsilon_E+\varepsilon_D+\varepsilon_1+\frac1m\Big).
\end{equation}
Moreover, by Axiom CEC$+(J)$ and~\eqref{def:cunifconst} for every $x\in H_i'$ we have
\begin{equation}\label{eq:conexpspelow}
	\ell_i K_5
	\le \log\,\lvert (f_{[\eta_0^i\ldots\,\eta_{\ell_i-1}^i]})'(x)\rvert
	\le \ell_i\log\, \D\,.
\end{equation}

To estimate the exponents in the horseshoe, we first look at the ``finite-time'' Lyapunov exponents. We will use the concatenate sequences
\[\begin{split}
	(\sigma_0^i\ldots\sigma_{m_i+\ell_i-1}^i)
	&\eqdef (\zeta^i_0\ldots\,\zeta^i_{m_i-1}\,\eta_0^i\ldots\eta_{\ell_i-1}^i)\\
	&= (\theta^i_1\ldots\theta^i_{r_i}\,\xi_0^i\ldots\xi_{m-1}^i\,
			\beta^i_1\ldots\beta^i_{s_i}
			\,\eta_0^i\ldots\eta_{\ell_i-1}^i),
\end{split}\]
for $i=1,\ldots,M$.
Observe that for $x\in I_i'$ we have
\[\begin{split}
	\log\,\lvert  (f_{[\sigma^i_0\ldots\,\sigma^i_{m_i+\ell_i-1}]})'(x)\rvert
	&= \log\,\lvert (f_{[\zeta^i_0\ldots\,\zeta^i_{m_i-1}]})'(x)\rvert
		+ \log\,\lvert (f_{[\eta^i_0\ldots\,\eta^i_{\ell_i-1}]})'
			(f_{[\zeta^i_0\ldots\,\zeta^i_{m_i-1}]}(x))\rvert\\
\text{by~\eqref{eq:minmax} and~\eqref{eq:conexpspelow}}\quad
	&\ge -\log \widehat K_0 +m(\alpha-\varepsilon_E-\varepsilon_D)+ \ell_i	K_5\\
\text{by~\eqref{eq:ellell2}}\quad
	&\ge -\log \widehat K_0 +m(\alpha-\varepsilon_E-\varepsilon_D)
		+ \Big(K_2\lvert\log\delta_0\rvert +K_3\Big)K_5\\
	\text{reordering}\quad&= m\alpha
		+ K_2K_5\lvert\log\delta_0\rvert
		- m(\varepsilon_E+\varepsilon_D)
		-\log \widehat K_0
		+ K_3K_5	\\
\text{by~\eqref{defcond:Kfinal}}\quad
	&\ge m\alpha+\vartheta	.	
\end{split}\]

By the above and~\eqref{eq:formulaa} we get%
\footnote{Recall that the hypotheses of the theorem permit that $\alpha=0$.}
\begin{eqnarray}
	\frac{1}{m_i+\ell_i}\log\,\lvert (f_{[\sigma^i_0\ldots\,\sigma^i_{m_i+\ell_i-1}]})'(x)\rvert
	&\ge& \frac{m}{m_i+\ell_i}\alpha+\frac{1}{m_i+\ell_i}\vartheta
	\notag\\
	&>&\max\Big\{0, \alpha -O(\varepsilon_E+\varepsilon_D+\varepsilon_1+\frac1m)\Big\}.\label{eq:fintimlya-up}
\end{eqnarray}
This provides the lower bound of the finite-time Lyapunov exponent.

On the other hand, acting as above, we have
\[\begin{split}
	\log\,\lvert & (f_{[\sigma^i_0\ldots\,\sigma^i_{m_i+\ell_i-1}]})'(x)\rvert\\
	&= \log\,\lvert (f_{[\zeta^i_0\ldots\,\zeta^i_{m_i-1}]})'(x)\rvert
		+ \log\,\lvert (f_{[\eta^i_0\ldots\,\eta^i_{\ell_i-1}]})'(f_{[\zeta^i_0\ldots\,\zeta^i_{m_i-1}]}(x))\rvert\\
\text{by~\eqref{eq:minmax} and~\eqref{eq:conexpspelow}}	&\le \log \widehat K_0 +m(\alpha+\varepsilon_E+\varepsilon_D)+ \ell_i \log \,\D\\
\text{by~\eqref{eq:ellell2}}
	&\le \log \widehat K_0 +m(\alpha+\varepsilon_E+\varepsilon_D)+
	\\
	&\phantom{\le}	2\Big(K_2\big(\lvert\log(\delta_0\widehat K_0^{-2})\rvert+2m(\varepsilon_E+\varepsilon_D)\big)+K_3\Big)\log \,\D\\
	&\le m\alpha 	+ m\,O(\varepsilon_E+\varepsilon_D+\varepsilon_1+\frac1m).
\end{split}\]
Observing that $m\le m_i+\ell_i$, with~\eqref{eq:formulaa} we obtain
\begin{equation}\label{eq:fintimlya-low}
	\frac{1}{m_i+\ell_i}\log\,\lvert (f_{[\sigma^i_0\ldots\,\sigma^i_{m_i+\ell_i-1}]})'(x)\rvert
	\le  \alpha +O(\varepsilon_E+\varepsilon_D+\varepsilon_1+\frac1m).
\end{equation}
This provides the upper bound of the finite-time Lyapunov exponent.

With the right choices of $\varepsilon_E,\varepsilon_D,m$ and $\delta_0$ in Section~\ref{sec:quantify}, the bounds in~\eqref{eq:fintimlya-up} and~\eqref{eq:fintimlya-low} will each be positive and $\lambda$-close to $\alpha$.

It remains to observe that we have the estimates~\eqref{eq:fintimlya-low},~\eqref{eq:fintimlya-up}  at each point in $I_i'$.  Hence, it is immediate that at every point in the horseshoe $\Gamma$ (and hence in its symbolic extension $\widehat\Gamma$) the Lyapunov exponent is between $\alpha-\lambda$ and $\alpha+\lambda$ for $\lambda=O(\varepsilon_E+\varepsilon_D+\varepsilon_1+1/m)$.

\subsection{Controlling entropy}\label{ssec:entropying}

Recall the construction of a multi-variable-time horseshoe in Section~\ref{sec:sketch}. We have obtained a multi-variable-time horseshoe $\Gamma$ and its symbolic extension $\widehat\Gamma\subset\Sigma_k\times\bS^1$, which is a compact $F$-invariant  set.

Observe that the transition times in the horseshoe $t_{ij}=m+s_i+\ell_i+r_j$ vary between numbers $t_{\rm min}\ge m+1$ and $t_{\rm max}\le m+S(m)$, where
\[
	S(m)
	\eqdef m_{\rm b} +m_{\rm f}+\max_i\ell_i.
\]

As $(F,\widehat\Gamma)$ symbolically extends $(T,\Gamma)$, we have
\[
	h_{\rm top}(F,{\widehat\Gamma})= h_{\rm top}(T,\Gamma).
\]
By Proposition~\ref{prop:wilvarhorse}, we have
\[
	h_{\rm top}(T,\Gamma)
	\ge \frac{\log M-\log(t_{\rm max}-t_{\rm min}+1)}{t_{\rm max}}
	\ge \frac{\log M-\log S(m)}{m+S(m)}
\]
and 
\[
	h_{\rm top}(T,\Gamma)
	\le \frac{1}{t_{\rm min}}\log M
	\le \frac{1}{m+1}\log M.
\]
It follows from~\eqref{eq:ellell2} that 
\[
	S(m)
	={m}\, O(\varepsilon_E+\varepsilon_D+\varepsilon_1+\frac1m).
\]
Hence,  from~\eqref{eq:saccoentropy} we get
\[
	\frac{\log{M}-\log{S(m)}}{m+S(m)}
	\ge h-\varepsilon_H - O(\varepsilon_E+\varepsilon_D+\varepsilon_1+\frac{\log m}{m}).
\]
and 
\[
	\frac{1}{m+1}\log M
	\le 
	h+\varepsilon_H+O(\frac1m).
\]

With the right choices of $\varepsilon_E,\varepsilon_D,m$ in Section~\ref{sec:quantify},  this gives the estimates for the topological entropy of $F|_{\widehat\Gamma}$.

\subsection{Proof of  Theorem~\ref{teo:PES1}}
It is now a consequence of Sections~\ref{sec:sketch}--\ref{ssec:entropying}.

\subsection{Controlling Birkhoff  averages -- Proof of  Theorem~\ref{teo:PES2}}\label{ssec:birkhoff}

We continue to consider the ingredients of Sections~\ref{sec:sketch}--\ref{ssec:entropying}, where now we consider some measure $\mu\in\cM_{\rm erg}$ and assume that the  Skeleton$\ast$ property holds relative to $J$ and $\mu$.  We  take $h=h(\mu)$ and $\alpha=\chi(\mu)$.

 Pick a countable 
 dense subset  $\{\psi_i\}_{i\ge1}$ of continuous (nonzero) functions in the space of all continuous functions on $\Sigma_k\times\bS^1$. Recall that in the space of invariant probabilities on $\Sigma_k\times\bS^1$ the following  function  $d\colon \cM\times\cM\to[0,1]$
	\begin{equation}\label{distance}
		d(\mu,\nu)\eqdef \sum_{i=1}^\infty
		2^{-i}\frac{1}{2\lVert\psi_i\rVert_\infty}
		\,\Big\lvert \int\psi_i\,d\mu-\int\psi_i\,d\nu\Big\rvert\,,
	\quad
	\lVert \psi\rVert_\infty\eqdef\sup\,\lvert\psi\rvert	
	\end{equation}	
provides a metric which induces the weak$\ast$ topology on $\cM$.

As in Section~\ref{sec:quantify}, we fix $\gamma\in(0,h)$ and $\lambda>0$ and we specify the other quantifiers as before. Especially important is $\varepsilon_E\ll\gamma$.

Fixing some preliminary constants, let $K$ be a positive integer satisfying
\begin{equation}\label{choicer00}
	2^{-K+1}< \frac \gamma 2
\end{equation}	
and choose $\gamma_0\in(0,\gamma)$ such that
\begin{equation}\label{choicer0}
	\gamma_0(1-2^{-K})\, \max_{i=1,\ldots,K}\lVert\psi_i\rVert_\infty^{-1}
	< \gamma.
\end{equation}
Moreover, assume that $K$ was chosen large enough such that $\{\psi_1,\ldots,\psi_K\}$ is $\gamma$-dense, that is, for every continuous $\varphi\colon \Sigma_k\times\bS^1\to\bR$ there exists $\psi_i$, $i\in\{1,\ldots,K\}$, such that
\[
	\lVert \psi_i-\varphi\rVert_\infty
	<\frac\gamma2.
\]
Choose also a number $\varepsilon_0>0$ sufficiently small such that  the modulus of continuity of each $\phi\in\{\psi_1,\ldots,\psi_K\}$ satisfies
\begin{equation}\label{e:modcon}
	\sup_{X\in\Sigma_k\times\bS^1}\sup_{Y\in B(X,\varepsilon_0)}
	\lvert \phi(Y)-\phi(X)\rvert <\gamma_0.
\end{equation}	

Now apply the Skeleton$\ast$ property to $J$ and $\mu$ and the finite family of functions $\psi_j$ and obtain a skeleton $\fX=\{X_i\}$ of points $X_i=(\xi^i,x_i)$. From item (vi) we obtain that for each $\phi\in\{\psi_1,\ldots,\psi_K\}$ it holds
\[
	-\frac{K_0}{m} - \varepsilon_E
	\le \frac1m \sum_{k=0}^{m-1}\phi(F^k(X_i)) - \int\phi\,d\mu
	\le \frac{K_0}{m} + \varepsilon_E.
\]

One can show that $\widehat\Gamma$ can be constructed in such a way that for every $j$ the Birkhoff averages of $\psi_j$ at every point in $\widehat\Gamma$ are between $\int\psi_j\,d\mu-\lambda$ and $\int\psi_j\,d\mu+\lambda$. The only difference is that we replace $\varepsilon_D$ with the modulus of continuity of $\psi_j$ taken at $\delta_0$ (and taken the maximum over all $\psi_j$).

Together with~\eqref{e:modcon}, the choice of $\varepsilon_E$, and Sections~\ref{ssec:selecting}--\ref{ssec:entropying} we obtain that for every $X\in\widehat\Gamma$ and each $\phi\in\{\psi_1,\ldots,\psi_K\}$
\[
	\int\phi\,d\mu - \frac\gamma2\le
	\liminf_{n\to\infty}\frac{1}{n}\sum_{k=0}^{n-1}\phi(F^k(X)) \le
	\limsup_{n\to\infty}\frac{1}{n}\sum_{k=0}^{n-1}\phi(F^k(X))
	\le\int\phi\,d\mu + \frac\gamma2\,.
\]
In particular, for every $F$-invariant probability measure $\nu$  supported on $\widehat\Gamma$
\[
	\Big\lvert\int\phi\,d\nu-\int\phi\,d\mu\Big\rvert < \gamma.
\]
Now recall that we concluded the above for all $\phi\in\{\psi_1,\ldots,\psi_K\}$. Hence, 
with $\gamma_0\in (0,\gamma)$, \eqref{distance}, \eqref{choicer00}, and \eqref{choicer0} we obtain
\[\begin{split}
	d(\nu,\mu)
	&\le
	\sum_{i=1}^K2^{-i}\frac{\gamma_0}{2\lVert\psi_i\rVert_\infty}
	+\sum_{i=K+1}^\infty
		2^{-i}\frac{1}{2\lVert\psi_i\rVert_\infty}
		\,\Big\lvert \int\psi_i\,d\mu-\int\psi_i\,d\nu\Big\rvert\\
	&\le (1-2^{-K})\frac{\gamma_0}{2}\max_{i=1,\ldots,K}\lVert\psi_i\rVert_\infty^{-1}
		+ 2^{-K}2
	< \gamma.
\end{split}\]
This proves the theorem.

\subsection{Perturbation from negative to positive exponents -- Proof of Theorem~\ref{teo:mimick}}

In this section, we give prove Theorem~\ref{teo:mimick}. The proof is very similar to the proofs of Theorems~\ref{teo:PES1} and~\ref{teo:PES2}, so we will only indicate the main points.
We also consider only the case $\alpha<0,\beta>0$, the other case is obtained reversing time.

In the case that $h(\mu)=0$ the following construction is almost identical with the only difference that we take, instead of an exponentially growing number of points and orbit pieces in the skeleton, a single point which will give rise to a periodic orbit (with topological entropy $0$). Hence, in the following, without loss of generality, we assume $h=h(\mu)>0$.

\subsubsection{Choosing quantifiers}
Given $h>0$, $\alpha<0$, $\beta>0$, fix $\gamma\in (0,h)$ and $\lambda>0$.

We let $\varepsilon_E, \varepsilon_H,\varepsilon_D$, each much smaller than $\lvert\alpha\rvert, \beta, \gamma, \lambda$, in particular we require $\alpha+\varepsilon_E<0$. Let $K_1,\ldots,K_4$ be the constants from Axiom CEC$+(J)$ and  $K_3',K_D$ the constants provided by Lemma~\ref{lem:disaxicec}. 

By the Skeleton$\ast$ property relative to $J$ and $\mu$, there are universal numbers $m_{\rm f},m_{\rm b}\ge1$ and constants $K_0,L_0\ge1$ and $n_0\ge1$ such that for every $m\ge n_0$ there is a skeleton $\fX=\{X_i\}$ with $X_i=(\xi^i,x_i)$ having the  properties that
\begin{equation}\label{eq:estentropii}
	M\eqdef \card\fX \ge L_0^{-1}e^{m(h-\varepsilon_H)},
\end{equation}
the finite sequences
  $(\xi_0^i\ldots\,\xi_{m-1}^i)$  are all different, for every $\ell\in\{1,\ldots,m\}$ it holds 
\begin{equation}\label{eq:useskeliii}
	K_0^{-1} e^{\ell(-\lvert\alpha\rvert-\varepsilon_E)}
	\le \lvert (f_{[\xi^i_0\ldots\,\xi^i_{\ell-1}]})'(x_i)\rvert
	\le K_0 e^{\ell(-\lvert\alpha\rvert+\varepsilon_E)},
\end{equation}
and there are finite sequences $(\theta_1^i\ldots\,\theta_{r_i}^i)$, $(\beta_1^i \ldots \,\beta_{s_i}^i)$, $r_i\le m_{\rm f},s_i\le m_{\rm b}$, so that
\[
	 x_i' \eqdef f_{[\theta_{r_i}^i\ldots\,\theta_1^i .]}(x_i)\in J,\quad
	 x_i'' \eqdef f_{[\xi_0^i\ldots\,\xi_{m-1}^i \beta_1^i \ldots \,\beta_{s_i}^i]}(x_i) \in J.
\]

\subsubsection{Choosing the Markov rectangles and transition times}
For $i=1,\ldots,M$ we now consider the intervals
\begin{equation} \label{eqn:hiprim}
	I_i'
	\eqdef B(x_i', e^{-m\beta})
	\quad\text{ and }\quad
	H_i'
	\eqdef
	f_{[\theta_1^i\ldots\,\theta_{r_i}^i
		\xi_0^i\ldots\,\xi_{m-1}^i
		\beta_1^i \ldots\, \beta_{s_i}^i]}(I_i').
\end{equation}
Let
\begin{equation}
\label{e.sigmai}
	(\sigma_0^i\ldots\sigma_{m_i-1}^i)
	\eqdef
	(\theta_1^i\ldots\theta_{r_i}^i\xi_0^i\ldots\xi_{m-1}^i \beta_1^i \ldots \beta_{s_i}^i),
	\quad m_i\eqdef r_i+m+s_i.
\end{equation}	

In the sequel we will use Corollary~\ref{cor:distortionC1} to control distortion. For that we fix $\delta_0>0$ small so that the modulus of continuity satisfies $\Mod(2\delta_0)<\varepsilon_D$.
Fix some small $\varepsilon\in(0,1)$ with $\varepsilon +\varepsilon_D< \beta$. We now also require  that $m$ is sufficiently large so that
\begin{equation}\label{eq:condmmm}
	e^{-m\beta}
	< \delta_0 \widehat K_0^{-1}
		e^{-(m_{\rm b}+m_{\rm f})(\varepsilon+\varepsilon_D)}
		e^{-m(\varepsilon+\varepsilon_D)},
		\quad\text{where}\quad
	\widehat K_0=K_0 \, \lVert F \rVert^{m_{\rm b}+m_{\rm f}}	.
\end{equation}

Applying the controlled covering in Lemma~\ref{lem:disaxicec} to the intervals $H_i'$, we get finite sequences $(\eta_0^i\ldots\,\eta_{\ell_i-1}^i)$ to cover some neighborhood of $J$:
\begin{equation}\label{eq:alltime}
	f_{[\eta_0^i\ldots\,\eta_{\ell_i-1}^i]}(H_i') \supset B(J, K_4)
	\quad\text{ with }\quad
	\ell_i \le K_2\,\lvert\log\,\lvert H_i'\rvert\rvert+K_3'.
\end{equation}
Like in Sections~\ref{sec:quantify}--\ref{ssec:entropying}, this lets us construct a multi-variable-time horseshoe $\Gamma$ and its symbolic extension $\widehat\Gamma$.
What remains is only to estimate the finite-time Lyapunov exponents and the entropy of this horseshoe.

\subsubsection{Controlling Lyapunov exponents}
We first provide an estimate for $m_i$ and $\ell_i$. Observe that $m_i$ is between $m$ and $m+m_{\rm b} + m_{\rm f}$. This, together with~\eqref{eq:useskeliii} and \eqref{eqn:hiprim}, gives us
\begin{equation}\label{eq:perfeito}
	\widehat K_0^{-1} e^{-m(\beta +\lvert\alpha\rvert + \varepsilon_D + \varepsilon_E)}
	\leq |H_i'|
	\leq \widehat K_0 e^{-m(\beta +\lvert\alpha\rvert - \varepsilon_D - \varepsilon_E)}.
\end{equation}
Hence, together with~\eqref{eq:alltime} we have
\begin{equation}\label{eq:telefone}
	\ell_i
	\le K_2 \log \widehat K_0
		+ K_2 m (\beta +\lvert\alpha\rvert + \varepsilon_D + \varepsilon_E)
		+ K_3'.
\end{equation}
On the other hand, since by~\eqref{eq:alltime} we cover the interval $J$, we get
\[
	 \lVert F \rVert ^{\ell_i}\lvert H_i'\rvert
	\ge \lvert f_{[\eta_0^i\ldots\,\eta_{\ell_i-1}^i]}(H_i') \rvert
	\ge \lvert J\rvert
\]
and with~\eqref{eq:perfeito} hence
\[\begin{split}
	\ell_i
	&\ge \frac{1}{\log  \lVert F \rVert }\big(
					-\log\,\lvert H_i'\rvert
					+\log\,\lvert J\rvert
					\big)\\
	&\ge \frac{1}{\log  \lVert F \rVert }\big(
					m (\beta +\lvert\alpha\rvert + \varepsilon_D + \varepsilon_E)
					-\log\widehat K_0
					-\lvert\log\,\lvert J\rvert\rvert
					\big)				.
\end{split}\]
This provides the following estimate
\begin{multline}\label{eq:fiiinal}
	\frac{1}{1+K_2(\beta+\lvert\alpha\rvert)} - O(\varepsilon_D+\varepsilon_E+\frac1m)
	\le \frac{m}{m_i+\ell_i}\\
	\le \frac{1}{1+\frac{1}{\log  \lVert F \rVert }(\beta+\lvert\alpha\rvert)} + O(\varepsilon_D+\varepsilon_E+\frac1m).
\end{multline}

By the controlled distortion by Lemma~\ref{lem:disaxicec} we have
\begin{equation}\label{eq:disto1}
	\dist f_{[\eta_0^i\ldots\,\eta_{\ell_i-1}^i]}|_{H_i'}
	\le \lvert H_i'\rvert^{-\varepsilon_D} K_D.
\end{equation}
What remains is to control the distortion along the whole trajectory. For that note that, recalling~\eqref{eqn:hiprim},
\[
	\widehat K_0^{-1} e^{-m(\beta +\lvert\alpha\rvert + \varepsilon_D + \varepsilon_E)}
		e^{m\beta}
	\le \frac{\lvert H_i'\rvert}{\lvert I_i'\rvert}
	\le \widehat K_0 e^{-m(\beta +\lvert\alpha\rvert - \varepsilon_D - \varepsilon_E)}
		e^{m\beta}.
\]

What remains is to control the distortion along the whole trajectory. For that we now use Corollary~\ref{cor:distortionC1}. First note that for every $\ell \in\{1,\ldots,m_i\}$ by~\eqref{eq:useskeliii} 
and the definition of $(\sigma_0^i\ldots\sigma_{m_i-1}^i)$ in \eqref{e.sigmai} 
and recalling the definition of $\widehat K_0$ in \eqref{eq:condmmm} 
we have
\[
	\lvert (f_{[\sigma_0^i\ldots\sigma_{\ell-1}^i]})'(x_i)\rvert
	\le \widehat K_0 e^{\ell (-\lvert\alpha\rvert+\varepsilon_E)}.
\]
We now can apply the corollary to the interval  
\begin{equation}\label{e.intervalIiprime}
	I_i'
	=B(x_i',e^{-m\beta})
	\subset Z
	\eqdef B(x_i',\delta_0\widehat K_0^{-1}e^{-m_i(\varepsilon+\varepsilon_D)}),
\end{equation}	
(recall~\eqref{eq:condmmm} to verify the inclusion)
obtaining
\begin{equation}\label{eq:disto2}
	\lvert \log\dist f_{{[\sigma_0^i\ldots\sigma_{\ell -1 }^i]}}|_{I_i'}\rvert \le \ell \varepsilon_D.
\end{equation}

We are now ready to estimate the finite-time Lyapunov exponents. Indeed, by the mean value theorem, there exists $y\in I_i'$ satisfying
\[
	\frac{\lvert J\rvert}{\lvert I_i'\rvert}
	\le \lvert (f_{[\sigma_0^i\ldots\,\sigma_{m_i-1}^i \eta_0^i \ldots\,\eta_{\ell_i-1}^i]})'(y)\rvert
	\le \frac{1}{\lvert I_i'\rvert}.
\]
Therefore, by distortion~\eqref{eq:disto1} and~\eqref{eq:disto2}, for every $x\in I_i'$ we get
\[
	\frac {|J|} {|I_i'|}
	\cdot e^{-m_i\varepsilon_D}
	 |H_i'|^{\varepsilon_D} K_D^{-1}
	 \le  |(f_{[\sigma_0^i\ldots\,\sigma_{m_i-1}^i \eta_0^i \ldots\,\eta_{\ell_i-1}^i]})'(x)|
	 \le \frac {1} {|I_i'|}
	 	\cdot e^{m_i\varepsilon_D}
		 |H_i'|^{-\varepsilon_D} K_D.
\]
Substituting \eqref{eqn:hiprim} and \eqref{e.intervalIiprime} and using~\eqref{eq:perfeito} 
and recalling $m_i\le m_{\rm f}+m+ m_{\rm b}$
we get
\[\begin{split}
	 |(f_{[\sigma_0^i\ldots\,\sigma_{m_i-1}^i \eta_0^i \ldots\,\eta_{\ell_i-1}^i]})'(x)|
	 &\le e^{m\beta}e^{m_i\varepsilon_D}
	 	\big(\widehat K_0e^{-m(\beta+\lvert\alpha\rvert-\varepsilon_D-\varepsilon_E)}\big)
			^{-\varepsilon_D}
	 	K_D \\
	&\le  K_D\widehat K_0^{-\varepsilon_D}e^{(m_{\rm f}+m_{\rm b})\varepsilon_D}
		e^{m\big(\beta +(\beta+\lvert\alpha\rvert + 1-\varepsilon_D-\varepsilon_E)\varepsilon_D\big)}		.
\end{split}\]
Analogously,
\[
	|(f_{[\sigma_0^i\ldots\,\sigma_{m_i-1}^i \eta_0^i \ldots\,\eta_{\ell_i-1}^i]})'(x)|
	 \ge
	 \lvert J\rvert
	 K_D^{-1}
	 \widehat K_0^{\varepsilon_D}e^{-(m_{\rm f}+m_{\rm b})\varepsilon_D}
		e^{m\big(\beta +(\beta+\lvert\alpha\rvert - 1+\varepsilon_D+\varepsilon_E)\varepsilon_D\big)}.	
\]
Summarizing, there is some constant $K>1$ so that
\begin{multline*}
	K^{-1}e^{m\big(\beta +(\beta+\lvert\alpha\rvert - 1+\varepsilon_D+\varepsilon_E)\varepsilon_D\big)}
	\le\\ |(f_{[\sigma_0^i\ldots\,\sigma_{m_i-1}^i \eta_0^i \ldots\,\eta_{\ell_i-1}^i]})'(x)|
	\le Ke^{m\big(\beta +(\beta+ \lvert\alpha\rvert+ 1-\varepsilon_D-\varepsilon_E)\varepsilon_D\big)}.
\end{multline*}
Substituting the bounds on $m/(m_i+\ell_i)$ in~\eqref{eq:fiiinal} we can bound  the finite-time Lyapunov exponents:
\begin{multline*}
	\frac{\beta}{1+K_2(\beta +\lvert\alpha\rvert)}
		- O(\varepsilon_D+ \varepsilon_E+ \frac1m) \\
	\leq \frac{1}{m_i+\ell_i}
	\log |(f_{[\sigma_0^i\ldots\,\sigma_{m_i-1}^i \eta_0^i \ldots\,\eta_{\ell_i-1}^i]})'(x)|
	\leq \frac{\beta}{1+\frac{1}{\log \,\lVert F\rVert}(\beta +\lvert\alpha\rvert)}
		+ O(\varepsilon_D+ \varepsilon_E+\frac1m).
\end{multline*}
This immediately implies the estimates of the  exponents in the horseshoe $\widehat\Gamma$.

\subsubsection{Controlling entropy}
As in the proof of Theorem~\ref{teo:PES1}, the topological entropy of $F$ on the horseshoe is estimated as follows.
Observe that the transition times in the horseshoe $t_{ij}=m+s_i+\ell_i+r_j$ vary between numbers $t_{\rm min}\ge m+1$ and $t_{\rm max}\le m+S(m)$, where
\[
	S(m)
	\eqdef m_{\rm b} +m_{\rm f}+\max_i\ell_i.
\]

As before, using Proposition~\ref{prop:wilvarhorse}, we have
\[
	h_{\rm top}(T,\Gamma)
	\ge \frac{\log M-\log S(m)}{m+S(m)}.
\]
It follows from~\eqref{eq:telefone} that 
\[
	S(m)
	\le mK_2(\beta+\lvert\alpha\rvert)
		+m\,  O (\varepsilon_E+\varepsilon_D+\varepsilon_1+\frac1m).
\]
Hence,  from~\eqref{eq:estentropii} we get
\[
	\frac{\log{M}-\log{S(m)}}{m+S(m)}
	\ge \frac{h-\varepsilon_H}{1+K_2(\beta+\lvert\alpha\rvert)}
		 - O(\varepsilon_E+\varepsilon_D+\varepsilon_1+\frac{\log m}{m}).
\]
This finishes the estimate of the entropy.

We finally explain the weak$\ast$ approximation of measures in $\widehat\Gamma$. We assume that the Skeleton $\fX=\{X_i\}$ was chosen such that, analogously to Section~\ref{ssec:birkhoff} in the proof of Theorem~\ref{teo:PES2}, the orbit of length $m$ starting in the points $X_i=(\xi^i,x_i)$ well approximate $\mu$. Recall again that the connecting times $r_i,s_i$ are bounded by some universal constants. Recall also that $\ell_i$ is of order $mK_2(\beta+\lvert\alpha\rvert)$. By construction, the set $\widehat\Gamma$ is built very close to these orbit pieces and hence any invariant measure supported on it has generic points always staying a fraction $K_2(\beta+\lvert\alpha\rvert)/(1-K_2(\beta+\lvert\alpha\rvert))$ of times close to them. This sketches the approximation of the measure and finishes the proof.

\section{Proofs of the main results}
\label{sec:mainproofs}

In this section we provide the still missing proofs of our main results.

Recall that, assuming CEC$\pm$ and Acc$\pm$, there is some closed interval satisfying the Axioms CEC$\pm(J)$ and Acc$\pm(J)$. Recall that for any $\mu\in\cM_{\rm erg}$, by Proposition~\ref{prop:lunch1} 
the map $F$ has the Skeleton$\ast$ property relative to $J$ and $\mu$. With this in mind, we now prove the theorems.

\subsection{Proof of Theorem~\ref{the:luzzaattoo}}
Consider
$\mu\in\cM_{\rm erg}$ with $\chi(\mu)=0$ and $h(\mu)>0$.
By the comments above we can apply
Theorem~\ref{teo:PES2} (to $J$ and $\mu$) to obtain measures with positive Lyapunov exponent which weak$\ast$ and in entropy approximate $\mu$. For the negative exponent measure, it is enough to apply Theorem~\ref{teo:PES2backward}.
\qed

\subsection{Proof of Theorem~\ref{theo:main3twin}}
Recall that by Lemma~\ref{lem:commonintJ} the axioms hold for any (sufficiently small) closed interval. Now it is enough to apply Theorem~\ref{teo:mimick} and recall the definition of $\Kdois(F)$ in~\eqref{eq:Kdois}.
\qed

\subsection{Twin measures -- Proof of 
Fact~\ref{pro:facttwinn}}\label{sec:twins}
First note that the skew-product map can be considered as the symbolic model of a $C^1$ diffeomorphism which has a dominated splitting in its tangent bundle into three bundles such that the central one corresponds to the fiber direction. In this setting, the Katok approximation by hyperbolic horseshoes applies
to  any given a hyperbolic measure,
 (see for instance~\cite{Cro:11}).
This implies that
the measure $\mu$ is a weak$\ast$ limit of invariant measures distributed on hyperbolic periodic orbits with fiber Lyapunov exponent close to $\chi(\mu)$. 

Given each such periodic point $X=(\xi,x)$ of period $p$, the iterated fiber map has  derivative $\lvert(f^p_\xi)'(x)\rvert<1$. Since we consider circle maps,  to each such point there exists a $p$-periodic point $Y=(\xi,y)$ satisfying $\lvert (f^p_\xi)'(y)\rvert\ge1$. Now the sequence of measures distributed on the corresponding  periodic orbits has a subsequence which converges weak$\ast$ to an invariant measure $\widetilde\mu$ satisfying $h(\widetilde\mu)=h(\mu)$. Indeed, the entropy of $\mu$ is determined by the entropy of the projected measure (compare~\eqref{eq:entropyproj}) and hence by the number of periodic points, only.
The so obtained measure $\widetilde\mu$ has nonnegative fiber exponent. The so obtained measure $\widetilde\mu$ is not necessarily ergodic. However, if $\mu$ is ergodic then its projection $\varpi_\ast\mu$ is ergodic, too. In this case, since the measure $\widetilde\mu$ has the same projection as $\mu$, any of its ergodic components has entropy equal to $h(\varpi_\ast\mu)=h(\mu)$. Finally, there is some ergodic component of $\widetilde\mu$ with nonnegative exponent.
\qed

\subsection{Proof of Theorem~\ref{theocor:varprinc}}
Denote $h\eqdef h_{\rm top}(F)$. By the usual variational principle for entropy, we have
\[
	h
	=\sup_{\mu\in \cM_{\rm erg}} h(\mu)
	=\max\{
		\sup_{\mu\in \cM_{\rm erg, <0}} h(\mu),
		\sup_{\mu\in \cM_{\rm erg, 0}} h(\mu),
		\sup_{\mu\in \cM_{\rm erg, >0}} h(\mu)\}>0.
\]	

Let us distinguish the three cases:
\begin{itemize}
\item [1)] $h=\sup_{\mu\in \cM_{\rm erg, 0}} h(\mu)$,
\item [2)] $h=\sup_{\mu\in \cM_{\rm erg, <0}} h(\mu)$,
\item [3)] $h=\sup_{\mu\in \cM_{\rm erg, >0}} h(\mu)$.
\end{itemize}

In Case 1 the assertion follows from Theorem~\ref{the:luzzaattoo}.

Cases 2 and 3 are analogous, we give the proof  of  Case 2.  By Fact~\ref{pro:facttwinn} for every ergodic measure $\mu\in\cM_{\rm erg,<0}$ with large entropy there exists an ergodic measure with equal entropy and exponent which is either zero or positive. In the former case we are in Case 1. In the latter one we are done.
\qed

\section{Examples}\label{sec:examples}

In this section we introduce a simple class of step skew-product maps as in~\eqref{eq:sp} whose fiber maps 
satisfy the
axioms in Section~\ref{sec:axioms}.
In this section we consider two types of examples (satisfying some open conditions):
blender-like examples
(Section~\ref{ss.viablenders}) and contraction-expansion-rotation examples first studied in~\cite{GorIlyKleNal:05} (Section~\ref{sec:Gorod}). Let us observe that although the nature of these two constructions is quite different
(although they share some common ingredients)  these properties
are essentially the same $C^1$-open and densely.
Let us also observe that the examples that we consider 
are robustly transitive step skew-product maps.
In this section we do not aim for full generality, but our goal is rather to present
simple constructions.

\subsection{Examples via blenders}
\label{ss.viablenders}

To construct examples of skew-product maps satisfying the axioms in our setting we begin by defining a blender of an iterated function system (this definition can be seen as a translation of the definition of a blender to the
one-dimensional context). In the next definition we also borrow and adapt the terminology 
commonly 
used for blenders (see, for example,~\cite[Chapter 6.2]{BonDiaVia:05}).

\begin{definition}[One-dimensional blenders]
\label{d.blenders}
Consider diffeomorphisms $f_0$, $\dots$, $f_{k-1}\colon \mathbb{S}^1\to \mathbb{S}^1$.
We say that the IFS $\{f_i\}$ has an \emph{expanding blender}
if there are finite sequences $(\xi_0\dots \xi_r)$ and $(\eta_0\dots \eta_\ell)$,
$\xi_i, \eta_j\in \{0,\dots, k-1\}$, such that the maps
$g_0=f_{[\xi_0\dots \,\xi_r]}$ and $g_1=f_{[\eta_0\dots\, \eta_\ell]}$ satisfy the following properties:
there are a number $\beta>1$, an interval $[a,b]\subset\mathbb{S}^1$, and points $c,d\in [a,b]$, $c<d$, such that:
\begin{enumerate}
\item (uniform expansion)
$g_0^\prime (x)\ge \beta$ for all $x\in [a,d]$ and
$g_1^\prime (x)\ge \beta$ for all $x\in [c,b]$,
\item (boundary condition)
$g_0(a)=g_1(c)=a$,
\item
(covering and invariance) $g_0([a,d])=[a,b]$ and $g_1([c,b])\subset [a,b]$
\end{enumerate}
(see Figure~\ref{fi.neublender}).
We say that $[a,b]$ is the \emph{domain of definition} of the blender and that $[c,d]$ is
 the \emph{superposition interval} of the blender.

The IFS $\{f_i\}$ is said to have a \emph{contracting blender} provided the IFS $\{f_i^{-1}\}$ has an expanding blender.
\end{definition}
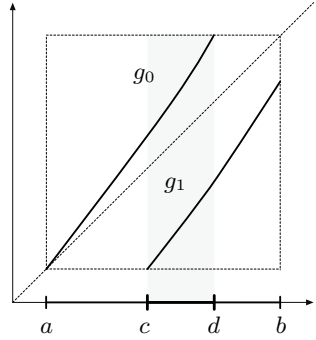
\begin{figure}
\begin{minipage}[h]{\linewidth}
\centering
 \begin{overpic}[scale=.35]{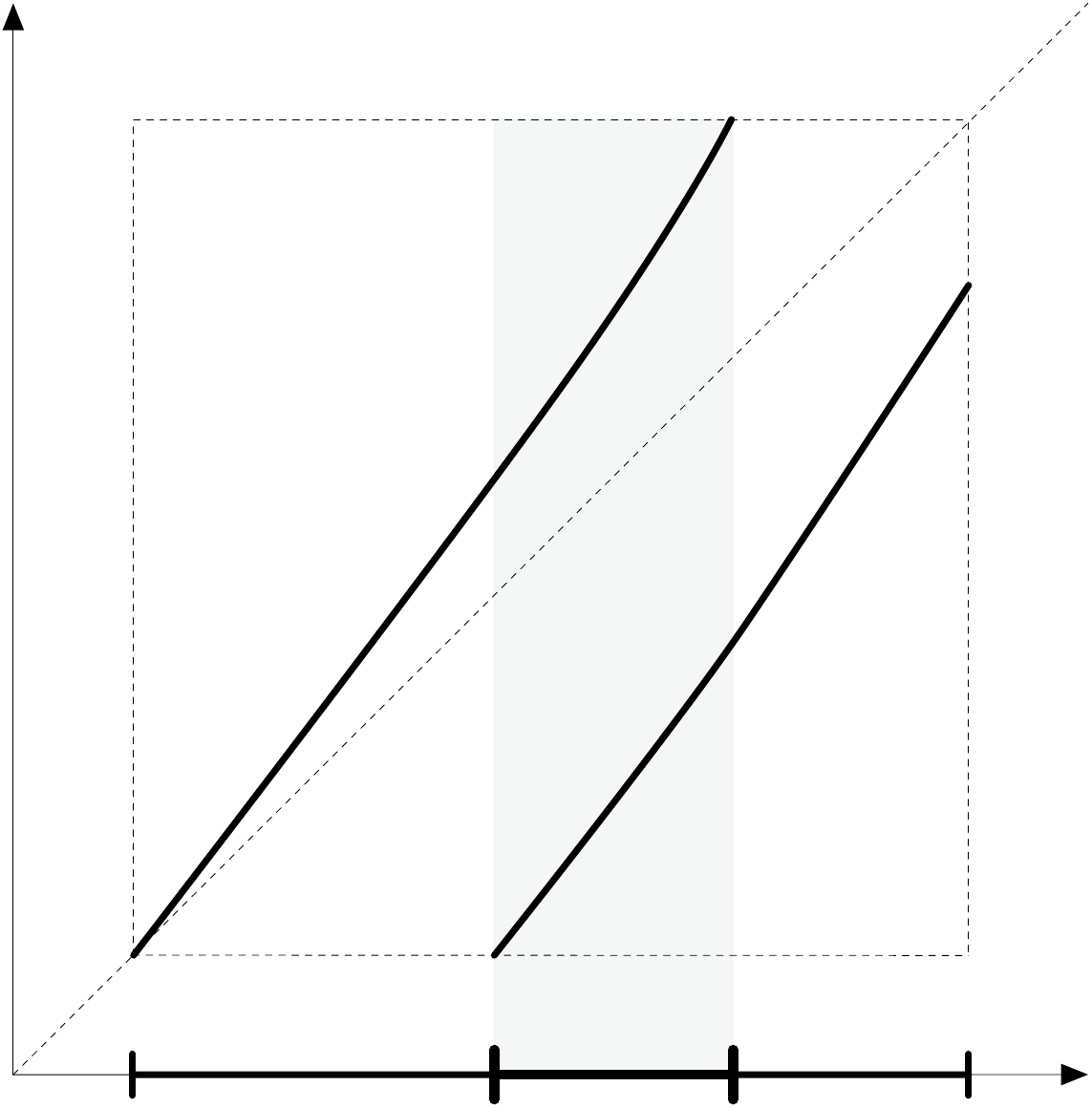}
			\put(40,75){\small{$g_0$}}
			\put(50,40){\small{$g_1$}}
  			\put(10,-7){\small$a$}
  			\put(86,-7){\small$b$}
  			\put(42,-7){\small$c$}
  			\put(64,-7){\small$d$}
  \end{overpic}
\caption{expanding blender}
\label{fi.neublender}
\end{minipage}
\end{figure}	

\begin{remark}
\label{r.continuation}
Following \cite{BonDia:96}, it is straightforward to see that the property of having a blender is an open property for the step skew-product.
Given an IFS $\{f_i\}$ of diffeomorphisms $f_0,\dots,f_{k-1}\colon \mathbb{S}^1\to \mathbb{S}^1$ having an expanding blender, then every family of maps $g_0,\dots,g_{k-1}\colon \mathbb{S}^1\to \mathbb{S}^1$ which are
$C^1$-close enough to $f_0,\dots,f_{k-1}$ associates an iterated function system which has an expanding blender
(where the elements in the definition of the blender depend continuously on the IFS).
\end{remark}

\begin{proposition}\label{p.example1}
Consider an IFS $\{f_i\}$ of diffeomorphisms $f_0,\dots,f_{k-1}\colon \mathbb{S}^1\to \mathbb{S}^1$ having
\begin{itemize}
\item[1.]
an expanding  blender with domain of definition $[a^+,b^+]$
and superposition interval $[c^+,d^+]$ and
\item[2.]
a contracting blender with domain of definition $[a^-,b^-]$ and superposition interval $[c^-,d^-]$.
\end{itemize}
Suppose that every point $x\in\bS^1$ has a forward and a backward
iterate in the interior of $[a^+,b^+]$ and $[a^-,b^-]$, respectively.

Then for every closed interval $J^+\subset (c^+,d^+)$, the IFS $\{f_i\}$ satisfies the
Axioms CEC$+(J^+)$ and Acc$\pm(J^+)$. For every closed interval $J^-\subset (c^-,d^-)$ the  IFS $\{f_i\}$ satisfies the Axioms CEC$-(J^-)$ and Acc$\pm(J^-)$.
\end{proposition}

We postpone the proof of the above proposition and derive first some consequences of it. In view of Remark~\ref{r.continuation} we have the following:

\begin{remark}
\label{r.robustblenders}
Consider an IFS $\{f_i\}$ of diffeomorphisms $f_0,\dots,f_{k-1}\colon \mathbb{S}^1\to \mathbb{S}^1$ satisfying the
hypotheses of Proposition~\ref{p.example1}.
Then every family of circle diffeomorphisms  $g_0,\dots,g_{k-1}$ that are
$C^1$-close enough to $f_0,\dots,f_{k-1}$ also satisfies these hypotheses.
\end{remark}

The following
is a standard  simple consequence of Proposition~\ref{p.example1}
whose proof is omitted. For the transitivity part see, for instance, the arguments in \cite[Section 5]{DiaGelRam:14} written using symbolic representations. The robustness follows from Remark~\ref{r.robustblenders}.

\begin{proposition}
Consider an IFS $\{f_i\}$ of  diffeomorphisms $f_0,\dots,f_{k-1}\colon \mathbb{S}^1\to \mathbb{S}^1$
satisfying the conditions of Proposition~\ref{p.example1}.
The the associated step skew-product map $F$ defined as in~\eqref{eq:sp} is (robustly) transitive.
\end{proposition}

\begin{proof}[Proof of Proposition~\ref{p.example1}]
We only prove the statement about an interval $J^+$, the other one is analogous.
Let $J=J^+$ be a closed interval in $(c,d)=(c^+,d^+)$ and let $[a,b]=[a^+,b^+]$ and $\beta=\beta^+$ a
corresponding expansion number.
We proof that Axioms CEC$+(J)$ and Acc$\pm(J)$ hold. 

We start with a preliminary construction.
 Let $\alpha\eqdef d-c$.
For any interval $I=I_0\subset [a,b]$ with length less that $\alpha$ we  either have
 $I\subset [a,d]$ or $I\subset [c,b]$. In the first case let $I_1=g_0(I_0)$, otherwise let
$I_1=g_1(I_0)$.
We call $I_1$ the \emph{successor} of $I_0$.
Arguing now inductively,  for $k\ge1$ let the interval $I_{k}$ be the successor of $I_{k-1}$, repeating this process as long as the interval $I_{k-1}$ is contained either in $[a,d]$ or in  $[c,b]$.

The expansion property 1. of the blender implies $|I_k|\ge \beta^k |I_0|$. Hence, there is a first $\iota$ such that every of the intervals $I_0,\dots, I_{\iota-1}$ is contained either in
 $[a,d]$ or in $[c,b]$, and that $I_\iota\supset [c,d]$.
Letting $\widetilde \beta>1$ be an upper  bound of the derivative of $g_0$ in $[a,d]$
 and of the derivative of $g_1$ in $[c,b]$, a straightforward calculation implies that
$$
	\frac{ \log \alpha -\log |I_0| }{\log \widetilde\beta}
	\le \iota\le \frac{ \log \alpha -\log |I_0| }{\log \beta}.
$$

Applying this construction now to $I_0=J$, there are a closed subinterval
$K\subset J$ and a finite sequence $(i_0\dots i_\iota)$ such that
$f_{[i_0\dots\, i_\iota]}( K)=[c,d]$. Further, we have $f_{[i_0\dots\, i_\iota \, 1]}( K) \supset[a,a+\beta\alpha]$ and hence there is a first integer $s$, independent of $J$ and $K$, such that
$f_{[i_0\dots\, i_\iota\,1\,0^s]}( K)\supset [a,d]$.
  This implies that with $n=\iota+2+s$
 $$
  \beta^n\le
 	(f_{[i_0\dots i_\iota\,1\,0^s]})' (x) \le {\widetilde\beta}^n
 $$
 and
$$
 \frac{\log (d-c) -\log |J|}{\log \widetilde \beta} + 2 + s\le
n \le \frac{\log (d-c) -\log |J|}{\log \beta} +2 + s.
$$
This immediately implies  Axiom CEC$+(J)$.

The above construction implies also that the orbit of any (nontrivial)
closed  subinterval of $(a,b)$ covers $[a,b)$. Indeed, just note that by construction
$a$ belongs to the interior of $g_{[i_0\dots i_\iota\,1]}(J)$.
This implies that the forward $g_0$ iterates of this interval covers $[a,b]$. We will summarize these remarks below (when arguing similarly for a contracting blender).

\begin{scholium}
\label{r.forwardandbackward}
The forward orbit of any nontrivial closed subinterval of $(a^+,b^+)$
for the IFS $\{f_i\}$
covers $[a^+, b^+]$.
The backward orbit of any nontrivial closed subinterval of $(a^-,b^-)$
for the IFS $\{f_i\}$
covers $[a^-, b^-]$.
\end{scholium}

Now, our hypothesis that every $x\in\bS^1$ has a backward iterate in the interior of $[a,b]=[a^+,b^+]$ together with the first part of Remark~\ref{r.forwardandbackward}
 imply Axiom Acc$+(J)$.

To see Axiom Acc$-(J)$ recall that every
$x\in\bS^1$ has a backward iterate $(a^-,b^-)$, thus there is a small
subinterval $K$ of $J$ having a backward iterate in $(a^-,b^-)$. The second part of
Scholium~\ref{r.forwardandbackward} implies  that the backward orbit of $K$
(thus of $J$) covers $[a^-,b^-]$.  By hypothesis, every
$x\in\bS^1$ has a forward iterate in $(a^-,b^-)$, which implies that
$x$ belongs to the backward orbit of $J$ and proves Axiom Acc$-(J)$. This completes the proof of the proposition
for closed subintervals of $(c^+,d^+)$.
\end{proof}

The following is an immediate consequence of our constructions or by applying Lemma~\ref{lem:transssitivo}.

\begin{corollary}
Consider an IFS satisfying the hypotheses of Proposition~\ref{p.example1}.

Then the IFS satisfies Axiom T and there is an interval $J\subset\bS^1$ such that the IFS satisfies Axioms CEC$\pm(J)$ and Acc$\pm(J)$.
\end{corollary}

\subsection{Contraction-expansion-rotation examples}\label{sec:Gorod}

The hypotheses in the following proposition are motivated by the constructions in \cite[Theorem 2]{GorIlyKleNal:05}, where -- for simplicity of this exposition -- we replace
the assumption of forward minimality  in~\cite{GorIlyKleNal:05} by the existence of an irrational rotation. More general cases can be treated
by slight modifications of our arguments.

\begin{proposition} 
Consider an IFS $\{f_i\}$ of diffeomorphisms $f_0,\dots, f_{k-1}\colon \mathbb{S}^1\to \mathbb{S}^1$, $k\ge2$.
Assume that
there are finites sequences $(\xi_0\dots \xi_r)$, $(\eta_0\dots \eta_s)$,
and  $(\zeta_0\dots \zeta_t)$ such that
\begin{enumerate}
\item
	$f_{[\xi_0\dots \,\xi_r]}$ has an attracting fixed point and
	$f_{[\eta_0\dots \,\eta_s]}$ has a repelling fixed point,
\item
	$f_{[\zeta_0\dots \,\zeta_t]}$ is an irrational rotation.%
\footnote{This hypothesis can be replaced by the assumption that the IFS has a ``sufficiently dense'' orbit, for instance, that every point has forward and backward iterates in the basin of attraction and in the basin of repulsion of fixed points in item 1.}
\end{enumerate}

Then  there are intervals $J^+,J^-\subset\bS^1$
such that any IFS $\{g_i\}$ of maps  $g_0$, $\dots$, $g_{k-1}\colon \mathbb{S}^1\to \mathbb{S}^1$
$C^1$-close enough to $f_0,\dots, f_{k-1}$,
satisfies Axioms CEC$+(J^+)$ and Acc$\pm(J^+)$ and
Axioms CEC$-(J^-)$ and Acc$\pm(J^-)$.
\end{proposition}

%
%
%
%
%
%

\begin{proof}
For notational simplicity let us prove the proposition when $r=s=t=0$
and  $(\xi_0\ldots \xi_r)=0$, $(\eta_0\dots \eta_s)=0$,
and  $(\nu_0\dots \nu_t)=1$.  
The general case is similar. 

As in the proof of Proposition~\ref{p.example1}, we show Axioms CEC$+(J^+)$ and Acc$\pm(J^+)$ only,  Axioms CEC$-(J^-)$ and Acc$\pm(J^-)$ follow similarly.

We begin by selecting appropriate neighbourhoods of $f_0$ and $f_1$.
As the map
 $f_0$ has a  repelling fixed point there are an interval $J\subset \mathbb{S}^1$
 and a neighbourhood $\cV(f_0)$ of $f_0$ such that for every $g_0\in \cV(f_0)$ and
 $x\in J$ it holds
 $g_0'(x)>1+\varepsilon$.

 Since $f_1$ is an irrational rotation there are a neighbourhood $\cV(f_1)$ of $f_1$ and numbers
$m_0$ and $\ell_0$ such that
\begin{itemize}
\item
for every $g_1\in \cV(f_1)$ and every interval $A\subset \mathbb{S}^1$ of size less that $(1-\varepsilon/2)|J|$ there exists
$m=m(A)\le m_0$ such that $g_1^{m}(A)\subset J$,
\item
for every $g_1\in \cV(f_1)$ and every pair of  intervals
$B\subset \mathbb{S}^1$ with $|B|>(1-\varepsilon/2)|J|$ and $C\subset \mathbb{S}^1$ with
$|C|< (1-\varepsilon)|J|$
 there exists
$\ell=\ell (B,C) \le \ell_0$ such that $g_1^\ell(B)$ contains $C$.
\end{itemize}

Take now any pair of intervals $H$ and $I$ with
$|H|, |I| < (1-\varepsilon)|J|$ and any pair of maps
$g_0\in \cV(f_0)$ and $g_1\in \cV(f_1)$.
We will exhibit
a trajectory $\xi_0\ldots\xi_{j-1}$
with
 $j \leq K_2 |\log|H|| + K_3$ such that
 $g_{[\xi_0\ldots\,\xi_{j-1}]}(H) \supset I$ and
  $\log |g_{[\xi_0\ldots\,\xi_{j-1}]}'(x)| > K_5 j$ for every $x\in H$.
  This will imply the proposition.

The argument now goes as follows,
take any pair
of maps $g_0\in \cV(f_0)$ and $g_1\in \cV(f_1)$ and any pair
of intervals $H$ and $I$ with $|H|, |I| < (1-\varepsilon)|J|$.
Consider the number $m=m(H)$ associated to $H$ with
$g^{m}_1(H) \subset J$ and apply $g_0$ to $g^{m}_1(H)$.
In this way we get the interval $g_0 \circ g_1^m(H)$ whose
size is at least $(1+\varepsilon) |H|$.
If the resulting interval is shorter than $(1-\varepsilon/2)|J|$, we can (and will)  repeat the procedure.
If it is larger then we rotate it onto $I$, that is, we consider $g_1^{\ell}\circ g_0 \circ g_1^m(H)$ where
$\ell=\ell  (g_0 \circ g_1^m(H),I)$ is the number associated to
 $g_0 \circ g_1^m(H)$ and $I$. By definition,
 $g_1^{\ell}\circ g_0 \circ g_1^m(H)\supset I$.

At each step of this procedure we increase the size of the interval at least
by a factor $(1+\varepsilon)$, so the number $k$ of steps we need to cover
the interval $I$ satisfies
$$
k\leq \frac {|\log |H||} {\log (1+\varepsilon)}.
$$
 Each step of the procedure
 (when the size of the considered iterations of $H$ is less than $(1-\varepsilon/2)|J|$)
 takes at most $m_0+1$ iterations. Finally, when one finally gets (after at most $k$ steps)
 an interval of size at least $(1-\varepsilon/2)|J|$) one needs at most
 $\ell_0$ iterations at the end to cover $I$.
 Hence  the total number
 of iterations needed to cover $I$ is at most
  $$
  j\le  k(m_0+1)+\ell_0\le k(m_0+1+\ell_0)\le (\ell_0+1+m_0)
  \frac {|\log |H||} {\log (1+\varepsilon)}.
  $$
  Finally, the accumulated derivative at each point of $H$ is not smaller than
  $$
  (1+\varepsilon)^k \ge \big( (1+\varepsilon)^{1/(m_0+1+\ell_0)} \big)^j.
  $$
Therefore
  $$
  \log |g_{[\xi_0\ldots\xi_{j-1}]}'(x)| \ge K_5 j,
  \qquad
  K_5 =
   \log \big( (1+\varepsilon)^{1/(m_0+1+\ell_0)} \big) .
  $$
  This concludes the proof of the proposition.
  \end{proof}
\subsection{Robust transitivity: general comments}\label{sss.robusttransitivity}
We will explain why the hypotheses in Proposition~\ref{p.example1}
(and hence the Axioms CEC$\pm$ and Acc$\pm$) are very natural in the
robustly transitive setting (for  step skew-products with fiber $\mathbb{S}^1$).

For this we need to review some constructions in \cite{BonDiaUre:02}  (see also the extensions in \cite{RodRodUre:07}). These papers consider $C^1$-robustly transitive
and non-hyperbolic diffeomorphisms having periodic points of different indices
(dimension of the unstable direction) and a partially hyperbolic
 splitting $E^{ss}\oplus E^c\oplus E^{uu}$ with three non-trivial bundles
such that $E^c$ is one-dimensional,
$E^{ss}$ is uniformly contracting, and $E^{uu}$ is uniformly expanding.
In this setting, the strong stable foliation (tangent to $E^{ss}$) and
the strong unstable foliation (tangent to $E^{uu}$) are well defined. In the case when there exists a foliation by circles tangent to $E^c$, \cite[Theorem 1.6]{BonDiaUre:02} claims that there is an open and dense
subset of those systems whose strong stable and unstable foliations are both minimal (every leaf is dense).

The density of the strong  unstable (strong stable) foliation translates to the skew-product setting as follows.
Every point has a forward (backward) orbit which is dense in $\mathbb{S}^1$ by the underlying IFS. This immediately translates to  the accessibility conditions
Acc$\pm$ that holds for any non-trivial interval of the circle.

The second ingredient of \cite{BonDiaUre:02} is the existence of blenders. Without giving all the details, we note that
the minimality of the strong stable foliation is guaranteed by the existence  of a finite family of
center-unstable blenders that intersect nicely every leaf of the strong stable and strong unstable foliations.
This property  turns out to be robust and involves only leaves of bounded size. A similar condition is used
to guarantee the minimality of the strong unstable foliation, now considering center-stable blenders.

Let us observe that blenders are just a special type of hyperbolic set satisfying some geometrical properties (roughly, a superposition-like property). Definition~\ref{d.blenders} just translates the notion of a blender to the 
setting of skew-products.

The transitivity of the diffeomorphisms implies that all center-unstable blenders
are homoclinically related (their invariant manifolds intersect cyclically). A similar assertion holds for center-stable blenders. This homoclinic relation between blenders
implies that it is enough to consider just one center-unstable blender and one center-stable blender, exactly as in Proposition~\ref{p.example1}.

The property of the
orbit of the strong unstable and strong stable leaves intersect nicely the corresponding blenders translates to following the property:  every point has forward and backward iterates in the domain of definition of the one-dimensional blenders (both the contracting and the expanding).

In the following table we state a ``dictionary" of the terms involved:

\begin{center}
\begin{tabular}{p{6cm}|p{6cm}}
$C^1$-robustly transitive diffeormorphism & step skew-product map\\
\hrule & \hrule \\
$\bullet$ center-stable blender & $\bullet$ contracting blender \\
$\bullet$ center-unstabe blender & $\bullet$ expanding blender \\
$\bullet$ the unstable foliation crosses nicely the blender & $\bullet$ every point has a forward iterate in the interior of the domain of the blender \\
$\bullet$ the stable foliation crosses nicely the blender & $\bullet$ every point has a backward iterate in the interior of the domain of the blender
\end{tabular}
\end{center}
\medskip

Let us finally observe that the proof of \cite{BonDiaUre:02} can be translated
{\emph{mutatis mutandi}}
to prove that the hypotheses of
Proposition~\ref{p.example1} holds open and densely in the step skew-product setting for robustly transitive maps
having simultaneously periodic points with are fiber contracting and periodic points which are fiber expanding.

\begin{proposition}
Consider the set $\cS=\cS(\Sigma_k\times\bS^1)$, $k\ge1$, of all step skew-product maps $F$ as in~\eqref{eq:sp}
with $C^1$-fiber maps
which are robustly transitive and have periodic points of different indices.
Then there is an $C^1$-open and dense subset of $\cS$ consisting of step skew-products
satisfying Axioms T, CEC$\pm(J)$, and Acc$\pm(J)$ for some interval $J$
in $\mathbb{S}^1$.
\end{proposition}

Approximating general skew-products by step skew-products, one can get the following:

\begin{corollary}
Consider the set $\cT=\cT(\Sigma_k\times\bS^1)$ of all skew-product maps $F$ as in~\eqref{eq:sp} with $C^1$-fiber maps
which are robustly transitive and have periodic points of different indices.
Then there is an $C^1$-dense subset of $\cT$ consisting of step skew-products
satisfying Axioms T, CEC$\pm(J)$, and Acc$\pm(J)$ for some interval $J$
in $\mathbb{S}^1$.
\end{corollary}

\section*{Appendix. Entropy}

Let us recall the definition of topological entropy of a general set (i.e., not necessarily compact and invariant) following Bowen \cite{Bow:73}.

Consider a compact metric space $X$, a continuous map $F\colon X\to X$, a set $A\subset X$,  and a finite open cover $\mathscr{C} = \{C_1, C_2,\ldots, C_n\}$ of $X$. Given $U\subset X$ we write $U \prec \mathscr{C}$ if there is an index $j$ so that $U\subset C_j$, and $U\nprec\mathscr{C}$ otherwise.
Taking $U\subset X$ we define $n_{F,\mathscr{C}}(U)\eqdef 0$ if $U \nprec \mathscr{C}$, $n_{F,\mathscr{C}}(U)\eqdef \ell$ if $F^k(U)\prec \mathscr{C}$ for every $k\in \{0, \dots,  \ell-1\}$ and $F^\ell(U)\nprec\mathcal{C}$, and let $n_{F,\mathscr{C}}(U)\eqdef \infty$ otherwise. If $\mathcal U$ is a countable collection of open sets, for $d>0$ let
\[
	m(\mathscr C,d,\mathcal U)
	\eqdef  \sum_{U\in\mathcal U}e^{-d \,n_{F,\mathscr{C}}(U)}.
\]
Given a set $A\subset X$, let
\[
	m_{\mathscr{C}, d} (A)
	\eqdef  \lim_{\epsilon \to 0}\inf m(\mathscr C,d,\mathcal U),
\]
where the infimum is taken over all countable open covers $\mathcal U$ of $A$ such that $e^{-n_{F,\mathscr{C}}(U)}<\epsilon$ for each $U\in\mathcal U$.
The \emph{topological entropy} of $F$ on $A$ is
$$
	h_{\rm top}(F,A)
	\eqdef  \sup_{\mathscr{C}} h_{\mathscr{C}}(F,A) ,
	\quad\text{ where }\quad
	h_{\mathscr{C}}(F,A)
	\eqdef  \inf\{d\colon m_{\mathscr{C}, d}(A)=0\}.
$$
When $A=X$, we simply write $h_{\rm top}(F) = h_{\rm top}(F,X)$.
In~\cite[Proposition 1]{Bow:73}, it is shown that  in the case of a compact set $Y$ this definition is equivalent to the canonical definition of topological entropy (see, for example, \cite[Chapter 7]{Wal:82}).

Recall that entropy is \emph{countably stable}, that is, for every countable family of sets $A_1,A_2,\ldots\subset X$ we have
\begin{equation}\label{eq:entcounstab}
	h_{\rm top}\big(F,\bigcup_{i\ge1}A_i\big)	
	=\max_{i\ge1}h_{\rm top}(F,A_i).
\end{equation}

Recall also the following result for factor maps. Let $Y$ be a compact metric space and let $G\colon Y\to Y$ be a continuous map. Assume that $G$ is a (topological) factor of $F$, that is, assume there exists a continuous surjective map $\varpi\colon X\to Y$ such that $\varpi\circ F=G\circ \varpi$. Then by~\cite{LedWal:77}
\begin{equation}\label{eq:entropyproj}
	\sup_{\mu\colon\mu\circ\varpi^{-1}=\nu}h(\mu)
	= h(\nu)+\int_Y h_{\rm top}(F,\varpi^{-1}(\xi))\,d\nu(\xi).
\end{equation}
Observe that in the case $Y=\Sigma_k\times\bS^1$, $X=\Sigma_k$, and $\varpi(\xi,x)=\xi$, for every $\xi\in\Sigma_k$ we have $h_{\rm top}(F,\varpi^{-1}(\xi))=0$.

\bibliographystyle{amsplain}

\end{document}